\documentclass[a4paper,11pt]{article}
\usepackage{fullpage}

\usepackage[utf8]{inputenc}
\usepackage[english]{babel} 

\usepackage{amsfonts, amssymb, amsbsy, amsmath, amsthm}
\usepackage{dsfont}

\usepackage[hyphens]{url}
\usepackage[colorlinks=true,urlcolor=blue]{hyperref}
\usepackage[capitalise]{cleveref}

\usepackage{algorithm2e}
\RestyleAlgo{ruled}
\Crefname{algocf}{Algorithm}{Algorithms}
\Crefname{equation}{Equation}{Equations}
\Crefname{figure}{Figure}{Figures}

\usepackage{pgfplots}
\usepackage{caption}
\usepackage{wrapfig}
\usetikzlibrary{positioning}
\usetikzlibrary{patterns}

\usepackage{setspace}

\usepackage{multicol}
\usepackage{stmaryrd}
\newtheorem{theorem}{Theorem}[section]
\newtheorem{lemma}[theorem]{Lemma}
\newtheorem{definition}[theorem]{Definition}

\newtheorem{corollary}[theorem]{Corollary}

\newtheorem{proposition}[theorem]{Proposition}

\theoremstyle{remark}

\newtheorem{remark}[theorem]{Remark}

\newcommand{\mA}{\mathcal{A}}
\newcommand{\mB}{\mathcal{B}}
\newcommand{\mK}{\mathcal{K}}

\newcommand{\mX}{\mathcal{X}}
\newcommand{\mF}{\mathcal{F}}

\newcommand{\mT}{\mathcal{T}}

\newcommand{\mM}{\mathcal{M}}
\newcommand{\mS}{\mathcal{S}}

\newcommand{\RR}{\mathbb{R}}
\newcommand{\NN}{\mathbb{N}}
\newcommand{\ZZ}{\mathbb{Z}}

\newcommand{\FF}{\mathbb{F}}

\newcommand{\bs}{{\bf s}}

\newcommand{\emphdef}[1]{{\emph{#1}}}

\newcommand{\M}{\operatorname{M}}
\newcommand{\Mon}{\operatorname{Mon}}

\def\squaresize{0.4}
\def\padding{0.1}

\newcommand{\drawsquares}[1]{
  \begin{tikzpicture}
    \foreach \color [count=\i] in {#1} {
    \draw[line width=0.5pt, fill=\color] (\i*\squaresize+\i*\padding,0) rectangle ++(\squaresize,\squaresize);
    }
  \end{tikzpicture}
}

\newcommand{\drawemphsquares}[1]{
  \begin{tikzpicture}
    \foreach \color [count=\i] in {#1} {
      \ifnum\i=1
        \draw[line width=1.5pt, fill=\color] (\i*\squaresize+\i*\padding,0) rectangle ++(\squaresize,\squaresize);
      \else
        \draw[line width=0.5pt, fill=\color] (\i*\squaresize+\i*\padding,0) rectangle ++(\squaresize,\squaresize);
      \fi
    }
  \end{tikzpicture}
}

\title{The Rado Multiplicity Problem\\in Vector Spaces over Finite Fields}
\author{
    Juanjo Ru\'e\thanks{
        Departament de Matemàtiques and Institut de Matemàtiques (IMTech) de la Universitat Politècnica de Catalunya (UPC), and Centre de Recerca Matemàtica (CRM), Barcelona,
Spain E-mail: {\tt juan.jose.rue@upc.edu}.
    } \and
    Christoph Spiegel\thanks{
    	Zuse Institute, Department for AI in Society, Science, and Technology and The Berlin Mathematics Research Center MATH+, Berlin, Germany. E-mail: {\tt spiegel@zib.de}.
    } 
}
\begin{document}
\allowdisplaybreaks

\maketitle

\begin{abstract}
    We study an analogue of the Ramsey multiplicity problem for additive structures, in particular establishing the minimum number of monochromatic $3$-APs in $3$-colorings of $\FF_3^n$ as well as obtaining the first non-trivial lower bound for the minimum number of monochromatic $4$-APs in $2$-colorings of $\FF_5^n$. The former parallels results by Cumings et al.~\cite{CummingsEtAl_2013} in extremal graph theory and the latter improves upon results of Saad and Wolf~\cite{SaadWolf_2017}. The lower bounds are notably obtained by extending the flag algebra calculus of Razborov~\cite{razborov2007flag} to additive structures in vector spaces over finite fields.
\end{abstract}

\section{Introduction}

In 1959 Goodman~\cite{Goodman_1959} proved that asymptotically at least a quarter of all vertex triples in any graph must either form a clique or an independent set. This lead to the study of the \emph{Ramsey multiplicity problem}, where one would like to determine the minimum number of monochromatic cliques of prescribed size over any edge-coloring of the complete graph. Ramsey's theorem combined with a basic double-counting argument guarantees that asymptotically some positive fraction of cliques needs to be monochromatic, so one is primarily concerned with determining the precise value of that fraction. Given that the value in Goodman's result is attained by, among other constructions, the random graph $G(n, 1/2)$, Erd\H{o}s~\cite{Erdos_1962} originally conjectured in 1962 that this should hold not just for triangles but for cliques of arbitrary size $k$. This was disproven when Thomason~\cite{Thomason_1989} in 1989 obtained an asymptotic upper bound of $0.936 \cdot 2^{1 - k(k-1)/2}$ for any $k \geq 4$. Somewhat more recently, Conlon~\cite{conlon2007ramsey} established a lower bound of $C^{-k^2 \, (1+o(1))}$ for $C \approx 2.18$. Getting improved asymptotic bounds or determining precise values beyond those of triangles has proven surprisingly challenging, with a long list of results~\cite{Thomason_1989, FranekRodl_1993, Thomason_1997, Giraud_1979, Niess_2012, Sperfeld_2011, GrzesikEtAl_2020, even2015note, parczyk2022new} slowly improving both upper and lower bounds for the specific cases of $k=4$ and $k=5$, but so far falling short of closing the gap.

Recently there has been an increased interest in studying the arithmetic analogue of this type of question, originally initiated when Graham, R\"odl, and Ruczinsky~\cite{GrahamRodlRucinski_1996} gave an asymptotic lower bound for the minimum number of monochromatic Schur triples in $2$-colorings of the first $n$ integers in 1996. The precise asymptotic minimum was determined not much later by Robertson and Zeilberger~\cite{RobertsonZeilberger_1998}, Schoen~\cite{Schoen_1999}, and Datskovsky~\cite{Datskovsky_2003}. The latter in fact also showed the surprising property that the number of Schur triples in any $2$-coloring of $\ZZ_n$ depends only on the cardinalities of the color classes, an observation extended by Cameron, Cilleruelo, and Serra~\cite{CameronCillerueloSerra_2007} to solutions of arbitrary linear equations in an odd number of variables in any finite abelian  group. 
Along similar lines, Parrilo, Robertson and Saracino~\cite{parrilo2008asymptotic} established upper and lower bounds for the minimum number of monochromatic $3$-APs in $2$-colorings of the first $n$ integers, in particular disproving a folklore conjecture that the minimum should be attained by a random coloring.

In this paper we will focus on the analogue of the Ramsey multiplicity problem for specific additive structures in vector spaces over finite fields of small order. We let $q \in \NN$ be a fixed prime power throughout and write $\FF_q$ for the unique finite field with $q$ elements, often also denoted by $\text{GF}(q)$. Given a subset $T \subseteq \FF_q^n$ and a linear map $L$ defined by some matrix $A \in \mM^{r \times m}(\ZZ)$ with integer entries co-prime to $q$\footnote{The methods presented here are certainly also applicable when the coefficients lie in $\FF_q$, as for example considered by Saad and Wolf~\cite{SaadWolf_2017}, Fox, Pham, and Zhao~\cite{FoxPhamZhao_2021}, and Kam\v{c}ev, Liebenau, and Morrison~\cite{KamcevLiebenauMorrison_2021a, KamcevLiebenauMorrison_2021b}, or even more generally when they define arbitrary permutations of $\FF_q$, as for example considered by Cameron, Cilleruelo, and Serra~\cite{CameronCillerueloSerra_2007}.}, we are interested in studying 
\begin{equation}
    \mS_L(T) = \{\bs = (s_1, \ldots, s_m) \in T^m : L(\bs) = A \cdot \bs= {\bf 0} \textrm{ and } s_i \neq s_j \textrm{ for } i \neq j  \},
\end{equation}
that is the set of solutions with all-distinct entries in $T$. Throughout, we will assume that $A$ is of full rank and that $\mS_L(\FF_q^n) \neq \emptyset$. We will also write $s_L(T) = |\mS_L(T)| / |\mS_L(\FF_q^n)|$. Arguably the most relevant additive structures are \emphdef{Schur triples}, that is the linear map $L_\text{Schur}$ satisfying  $\mS_{L_\text{Schur}}(T) = \{ (x, y, z) \in T^3 : x + y = z, x \neq y, x \neq 0, y \neq 0 \}$, as well as $k$-term arithmetic progressions ($k$-APs), that is the linear maps $L_{k\text{-AP}}$ for any $k \geq 3$ satisfying $\mS_{L_{k\text{-AP}}}(T) = \{ (a, a+d, \ldots, a+(k-1)\,d) \in T^k: d \neq 0 \}$.

Writing $[c] = \{1, \ldots, c\}$ for some given number of colors $c \in \NN$, we call $\gamma: \FF_q^n \to [c]$ a \emphdef{$c$-coloring} of dimension $\dim(\gamma) = n$ and write $\gamma^{(i)}$ for the set of elements colored with $1 \leq i \leq c$ as well as $\Gamma_{q,c} (n)$ for the set of all $c$-colorings of $\FF_q^n$. The \emph{Rado multiplicity problem} (to coin a term analogous to that used in graph theory) is concerned with determining
\begin{equation}
    m_{q,c} (L) = \lim_{n \to \infty} \, \min_{\gamma \in \Gamma_{q,c} (n)} s_L(\gamma^{(1)}) + \ldots + s_L(\gamma^{(c)}).
\end{equation}
The limit exists by monotonicity and we have $0 \leq m_{q,c} (L) \leq 1$ by definition. If the requirements of Rado's theorem~\cite{rado1933studien} are met, then we in fact have $m_{q,c} (L) > 0$ and we say that $L$ is \emph{$c$-common} for $q$ if $m_{q,c} (L) = c^{1 - m}$, where we recall that $L$ is defined by a matrix $A$ with $r$ rows and $m$ columns, that is if the minimum number of monochromatic solutions is attained in expectation by a uniform random coloring.
For $r = 1$ the result of Cameron, Cilleruelo, and Serra~\cite{CameronCillerueloSerra_2007} establishes that any $L$ is $2$-common if $m$ is odd.
When $m$ is even, Saad and Wolf~\cite{SaadWolf_2017} showed that any $L$ where the coefficients can be partitioned into pairs, with each pair summing to zero, is $2$-common. Fox, Pham, and Zhao~\cite{FoxPhamZhao_2021} showed that this sufficient condition is in fact also necessary. The case when $r > 1$ is much less understood, with Kam\v{c}ev, Liebenau, and Morrison~\cite{KamcevLiebenauMorrison_2021a} recently characterizing a large family of non-common linear maps by showing that any $L$ that {\lq}induces{\rq} some smaller $2 \times 4$ linear map is uncommon.

Focusing on specific values of $q$, Kr\'al, Lamaison, and Pach~\cite{lamaison2022common} also recently characterized the $2$-common $L$ for $q=2$ when $r=2$ and $m$ is assumed to be odd. When $q=5$, the most relevant additive structures to study is that of $4$-APs~\cite{wolf2010minimum}. Saad and Wolf~\cite{SaadWolf_2017} showed that they are not $2$-common by establishing an upper bound of $1/8 - 7 \cdot 2^{10} \cdot 5^{-2} \approx 0.1247 < 2^{-4}$.
We establish the first non-trivial lower bound for this problem along with a significantly improved upper bound.
\begin{proposition} \label{prop:4apresult}
    We have $1/10 < m_{5,2} (A_{4\textrm{-AP}}) \leq 13/126 = 0.1\overline{031746}$.
\end{proposition}
We would in fact conjecture that the upper bound is tight. Going beyond $4$-APs, we can also show that  $m_{5,2}(A_{5\textrm{-AP}}) \leq 1/126 < 2^{-4}$, establishing that $5$-APs are likewise not $2$-common in $\FF_5$, but in this case did not obtain any meaningful lower bound.

The study of monochromatic structures in colorings with more than two colors has also proven relevant in extremal graph theory. Most notably, Cummings et al.~\cite{CummingsEtAl_2013} extended the results of Goodman~\cite{Goodman_1959} by establishing the exact Ramsey multiplicity of triangles in $3$-colorings and showing that they are not $3$-common despite being $2$-common. We consider a similar question in the additive setting and establish the exact Rado multiplicity of $3$-APs in $3$-colorings of $\FF_3^n$, likewise showing that they are not $3$-common.
\begin{theorem} \label{thm:3ap3colorresult}
    We have $m_{3,3} (A_{3\textrm{-AP}}) = 1/27$.
\end{theorem}
We can also show that $0.041258 \leq m_{2,3}(A_\text{Schur}) = 1/16 < 3^{-2}$ as well as $m_{3,3} (A_\text{Schur}) \leq 7/81 < 3^{-2}$, establishing that Schur triples are also not $3$-common for $q = 2$ and $q = 3$.

\medskip

Upper bounds of all results are obtained through explicit blowup-type constructions. Lower bounds in the graph theoretic setting have recently been obtained through a computational approach relying on flag algebras due to Razborov~\cite{razborov2007flag}. This approach has been extended to different contexts, but so far seems to not have been explored in the arithmetic setting. We take a first step in that direction by developing the required theory in the finite-field model and applying it to obtain the above mentioned results.

\paragraph{Outline} We develop the necessary basic notions for colorings of vector spaces over finite fields and how to count solutions in them in \cref{sec:fundamentals}. We then formally introduce the flag algebras along with a basic approach for obtaining upper bounds through blow-up constructions in \cref{sec:flagalgebras}. Relying on the previous sections, we establish how the results presented in the introduction were obtained in \cref{sec:proofs}. Finally, \cref{sec:remarks} contains some concluding remarks and possible directions for future research.

\section{Isomorphisms and double counting} \label{sec:fundamentals}

Since $q$ will be an arbitrary but fixed prime power  and $c$ and arbitrary but fixed integer throughout this section as well as the following two, we will in general omit both from notation, so in particular we write $\Gamma (n) = \Gamma_{q,c} (n)$ for the set of all $c$-colorings of dimension $n$. We will also write  $\Gamma = \bigcup_{n = 0}^\infty \Gamma(n)$ and let the $0$-dimensional vector space consist of a single point, that is $\FF_q^0 = \{0\}$, and consequently $|\Gamma (0)| = c$ regardless of $q$.

\subsection{A notion of partially fixed morphisms for vector spaces}

Let us write $e_j$ for the $j$-th canonical unit basis vector of $\FF_q^n$ for $1 \leq j \leq n$ as well as, somewhat unconventionally but notationally convenient, $e_0$ for the zero vector.

\begin{definition}
    We refer to an affine linear map $\varphi: \FF_q^k \to \FF_q^n$ as a \emphdef{morphism} and say that it is \emphdef{$t$-fixed} for some $t \geq 0$ if $\varphi(e_j) = e_j$ for all $0 \leq j \leq t$. A morphism is a \emphdef{monomorphism} whenever it is injective and a monomorphism is an \emphdef{isomorphism} whenever $n = k$.
\end{definition}

For notational convenience, we extend the range of $t$ to $-1$ in order to include \emphdef{unfixed} morphisms and will always use $t^+$ to denote $\max\{t,0\}$. For a given $t \geq -1$ and $n, k \geq t^+$, we let $\M_t(k;n)$ denote the set of $t$-fixed morphisms from $\FF_q^k$ to $\FF_q^n$ up to $t$-fixed isomorphism of $\FF_q^k$. When $n \geq k$, we likewise write $\Mon_t(k;n)$ for the set of monomorphisms with the same properties.
We will also refer to the image of an element of $\Mon_t(k;n)$ as a \emphdef{$t$-fixed $k$-dimensional subspace} of $\FF_q^n$, so that $0$-fixed subspaces correspond to linear subspaces and unfixed subspaces correspond to affine subspaces. Let us generalize our notation a bit further: given $k_1, \ldots, k_m \geq t^+$ and $n \geq k_1 + \ldots + k_m - (m-1) \, t^+$, we let $\Mon_t(k_1, \ldots, k_m; n)$ denote the set of all tuples of monomorphisms $(\varphi_1, \ldots, \varphi_m) \in \Mon_t(k_1; n) \times \ldots \times \Mon_t(k_m; n)$ only overlapping in the $t$-fixed subspace, that is they satisfy 
\begin{equation*}
    \varphi_i(\FF_q^{k_i}) \cap \varphi_j(\FF_q^{k_j}) =
        \begin{cases}
        \operatorname{id}_{t,n} (\FF_q^t) & \text{ for } t \geq 0,\\
        \emptyset & \text{ for } t = -1,
        \end{cases}
\end{equation*}
for any $1 \leq i < j \leq m$, where we let $\operatorname{id}_{t,n}$ denote the unique $t$-fixed isomorphism from $\FF_q^t$ to $\FF_q^n$.

In order to count monomorphisms, we will need to define a $q$-analog of the multinomial coefficient in the form of the Gaussian multinomial coefficient.
Briefly assuming $q \geq 1$ to be real, we write $[n]_q=\sum_{i=0}^{n-1}q^i$ for the \emphdef{$q$-number} of $n$, where we note that $[n]_q = (1-q^n) / (1-q)$ for $q \neq 1$ and $[n]_1 = \lim_{q \to 1} [n]_q = n$, as well as $[n]_q! = [n]_q \cdots [2]_q \, [1]_q$ for the \emphdef{$q$-factorial} of $n$. The \emphdef{Gaussian multinomial coefficient} can now be defined for any $k_1, \ldots, k_m \geq 0$ and $n \geq k' = k_1 + \ldots + k_m$ as
\begin{equation*}
    \binom{n}{k_1, \ldots, k_m}_q = \frac{[n]_q!}{[k_1]_q! \, \cdots \, [k_m]_q! \, [n-k']_q!}.
\end{equation*}
\begin{lemma} \label{lemma:count_morphisms}
    For any integers $0 \leq k_1, \ldots, k_m$ and $n \geq k' = k_1 + \ldots + k_m$, we have
    \begin{equation*}
        |\Mon_{-1}(k_1, \ldots, k_m;n)| = q^{n - k'} \, \binom{n}{k_1, \ldots, k_m}_q.
    \end{equation*}
    For integers $0 \leq t \leq k_1, \ldots, k_m$ and $n \geq k' = k_1 + \ldots + k_m - (m-1) \, t$, we have
    \begin{equation*}
        |\Mon_t(k_1, \ldots, k_m; n)| = \binom{n-t}{k_1-t,\ldots,k_m-t}_q.
    \end{equation*}
\end{lemma}
\begin{proof}
    The Gaussian \emph{binomial} coefficient $\binom{n}{k}_q$ counts the number of $0$-fixed $k$-dimensional subspaces of $\FF_q^n$. To see that this is true, note that there are exactly $(q^n - 1) \, (q^{n-1} - 1) \cdots (q^{n-k+1} - 1)$ ways of choosing $k$ linearly independent vectors in $\FF_q^n$ and accounting for isomorphisms in $\FF_q^k$ gives the required denominator of $(q^k - 1) \, (q^{k-1} - 1) \cdots (q - 1)$. Noting that
    \begin{align*}
        |\Mon_t(k_1, \ldots, k_m;n)| & = \Mon_0(k_1-t, \ldots, k_m-t;n-t)| \\
        & = |\Mon_0(k_1-t;n-t)| \, |\Mon_0(k_2-t, \ldots, k_m-t;n-k_1)|
    \end{align*}
    and iterating gives the desired equality for $|\Mon_t(k_1,\ldots,k_m;n)|$ when $t \geq 0$. The formula for $t=-1$ follows by noting that $\FF_q^n$ contains exactly $q^{n-k'}$ affine copies of a given $k'$-dimensional subspace.
\end{proof}

\begin{corollary} \label{eq:mon_doble_counting}
    We have
    \begin{equation*} 
         |\Mon_t(k_1, \ldots, k_m; n')| \, |\Mon_t(n'; n)| = |\Mon_t(k_1, \ldots, k_m; n)| \, \binom{n-k'}{n' - k'}_q
    \end{equation*}
    for any $t \geq -1$, $k_1, \ldots, k_m \geq t^+$, and $n \geq n' \geq k' = k_1 + \ldots + k_m - (m-1) \, t^+$.
\end{corollary}

\subsection{Isomorphisms and densities for colorings}

We say two colorings $\gamma_1, \gamma_2 \in \Gamma(n)$ are \emphdef{$t$-fixed isomorphic} for some $t \geq -1$, denoted by $\gamma_1 \cong_t \gamma_2$, if there exists a $t$-fixed isomorphism $\varphi: \FF_q^n \to \FF_q^n$ satisfying $\gamma_1 \equiv \gamma_2 \circ \varphi$. We let
\begin{equation}
    \Gamma^{t}(n) = \Gamma(n) / \cong_{t}   
\end{equation}
denote the set of all $c$-colorings of $\FF_q^n$ up to $t$-fixed isomorphism, i.e., we pick one canonical but arbitrary representative for each equivalence class. We will also write $\Gamma^t = \bigcup_{n=t^+}^\infty \Gamma^t(n)$ for the set of all colorings of any dimension up to $t$-fixed isomorphism.
Given two colorings $\delta \in \Gamma^{t}(k)$ and $\gamma \in \Gamma^{t}(n)$ with $n \geq k \geq t^+$, we say that a $t$-fixed monomorphism $\varphi \in \M_t(k;n)$ \emphdef{induces a copy} of $\delta$ in $\gamma$ and that $\delta$ is a \emphdef{sub-coloring} of $\gamma$ if $\delta \cong_t \gamma \circ \varphi$. When replacing the monomorphism with a not-necessarily-injective morphism, we say that it induces a \emphdef{degenerate} copy and that $\delta$ is a \emphdef{degenerate} sub-coloring.

\begin{table}[h!]
\begin{center}
\begin{tabular}{ c | c c c c c }
$q / n$ & 1 & 2 & 3 & 4 & 5\\ \hline
 2 & 3 & 5 & 10 & 32 & 382 \\ 
 3 & 4 & 14 & 1028  &  &  \\  
 4 & 8 & 1648 \\
 5 & 6 & 3324 &  &  & 
\end{tabular}
\hspace{1em}
\begin{tabular}{ c | c c c c c }
$q / n$ & 1 & 2 & 3 & 4 & 5\\ \hline
 2 & 4 & 8 & 20 & 92 & 2744 \\ 
 3 & 6 & 36 & 15\,636 &  &  \\  
 4 & 14 & 7724 & \\
 5 & 12 & 72\,192 &  &  & 
\end{tabular}
\end{center}
\caption{Number of $2$-colorings of $\FF_q^n$ up to unfixed (left) and $0$-fixed (right) isomorphism. Note that the cardinality of $\Gamma_{2,2}^t(n)$ is respectively given by the OEIS sequences {\tt A000585} and {\tt A000214}.}\label{table:nr2colorings}
\end{table}

\begin{table}[h!]
\begin{center}
\begin{tabular}{ c | c c c c }
$q / n$ & 1 & 2 & 3 & 4 \\ \hline
 2 & 6 & 15 & 60 & 996 \\ 
 3 & 10 & 140 & 25\,665\,178 &   \\  
 4 & 30 & 1\,630\,868 \\
 5 & 24 & 70\,793\,574 &  & 
\end{tabular}
\hspace{1em}
\begin{tabular}{ c | c c c c }
$q / n$ & 1 & 2 & 3 & 4 \\ \hline
 2 & 9 & 30 & 180 & 6546  \\ 
 3 & 18 & 648 &  &    \\  
 4 & 69 & 8\,451\,708 \\
 5 & 72 &  &  & 
\end{tabular}
\end{center}
\caption{Number of $3$-colorings of $\FF_q^n$ up to unfixed (left) and $0$-fixed (right) isomorphism.}\label{table:nr3colorings}
\end{table}

Given $k_1, \ldots, k_m \geq t^+$ and $n \geq k_1 + \ldots + k_m - (m-1) \, t^+$, the \emphdef{density} of some colorings $\delta_1 \in \Gamma^t(k_1)$, ..., $\delta_m \in \Gamma^t(k_m)$ in $\gamma \in \Gamma^t(n)$ is now defined as the probability that a a tuple of $t$-fixed monomorphism chosen uniformly at random from $\Mon_t(k_1, \ldots, k_m; n)$ induces copies of $\delta_1, \ldots, \delta_m$ in $\gamma$, that is  
\begin{equation}
    p_t (\delta_1, \ldots, \delta_m; \gamma) = \frac{|\{ (\varphi_1, \ldots, \varphi_m)  \in \Mon_t(k_1,\ldots,k_m;n) : \gamma \circ \varphi_i \cong_t \delta_i \text{ for all } 1\leq i \leq m\}|}{|\Mon_t(k_1,\ldots,k_m;n)|}.
\end{equation}
For $n, k \geq t^+$, we also let the \emphdef{degenerate density} of some $\delta \in \Gamma^t(k)$ in $\gamma$ denote the probability that a not-necessarily-injective $t$-fixed morphism does the same, that is 
\begin{equation}
    p_t^d (\delta; \gamma) = |\{ \varphi \in \M_t(k;n) : \gamma \circ \varphi \cong_t \delta \}| \, / \, |\M_t(k;n)|.
\end{equation}
The following lemma states that asymptotically the two notions converge to each other.
\begin{lemma} \label{lemma:asymptoticdensity}
    For any $t \geq -1$, $k \geq t^+$, $\delta \in \Gamma^t(k)$ and $\gamma \in \Gamma^t(n)$ with $n \to \infty$, we have
    \begin{equation}
        p_t^d(\delta; \gamma) = p_t(\delta; \gamma) \, (1+o_n(1)).
    \end{equation}
\end{lemma}
\begin{proof}
    This follows immediately by noting that $|\M_t(k,n)| = |\Mon_t(k,n)| \, (1+o_n(1))$, that is as $n$ becomes large enough we expect the average morphism to be injective.
\end{proof}

The following now establishes a simple but fundamental averaging equality for the non-degenerate notion of density.
\begin{lemma} \label{lemma:colordoublecounting}
    For any $t \geq -1$, $k_1, \ldots, k_m \geq t^+$, $n \geq n' \geq k' = k_1 + \ldots + k_m - (m-1) \, t^+$, and colorings $\delta_1 \in \Gamma^t(k_1)$, ..., $\delta_m \in \Gamma^t(k_m)$, $\gamma \in \Gamma^t(n)$, we have
    \begin{equation}
        p_t(\delta_1, \ldots, \delta_m; \gamma) = \sum_{\beta \in \Gamma^t(n')} p_t(\delta_1, \ldots, \delta_m; \beta) \, p_t(\beta; \gamma).
    \end{equation}
\end{lemma}
\begin{proof}
   Underlying this statement is a basic double counting argument. We have
    \begin{align*} 
        & \phantom{=} p_t(\delta_1, \ldots, \delta_m; \gamma) \, |\Mon_t(k_1, \ldots k_m; n)| \, \binom{n-k'}{n'-k'}_q\\
        & = \sum_{(\varphi_1, \ldots, \varphi_m) \in \Mon_t(k_1, \ldots, k_m; n)} \mathds{1}_{\gamma \circ \varphi_1 \cong_t \delta_1} \cdots \mathds{1}_{\gamma \circ \varphi_m \cong_t \delta_m} \, \binom{n-k'}{n'-k'}_q \\
        & = \sum_{\varphi' \in \Mon_t(n', n)} \sum_{(\varphi_1, \ldots, \varphi_m) \in \Mon_t(k_1, \ldots, k_m; n')} \mathds{1}_{\gamma \circ \varphi' \circ \varphi_1 \cong_t \delta_1} \cdots \mathds{1}_{\gamma \circ \varphi' \circ \varphi_m \cong_t \delta_m} \\
        & = \sum_{\beta \in \mF^\tau_{n'}}  \sum_{\varphi' \in \Mon_t(n', n)} \mathds{1}_{\gamma \circ \varphi' \cong_t \beta} \sum_{(\varphi_1, \ldots, \varphi_m) \in \Mon_t(k_1, \ldots, k_m; n')} \mathds{1}_{\beta \circ \varphi_1 \cong_t \delta_1} \cdots \mathds{1}_{\beta \circ \varphi_m \cong_t \delta_m} \\
        & = \sum_{\beta \in \mF^\tau_{n'}} p_t(\beta; \gamma) \, |\Mon_t(n'; n)| \, p_t(\delta_1, \ldots, \delta_m; \beta) \, |\Mon_t(k_1, \ldots, k_m; n')|,
    \end{align*}
    establishing the statement through \cref{eq:mon_doble_counting}.
\end{proof}

\subsection{Counting solutions in colorings}
 
For a given linear map $L$, we have already defined the set $\mS_L(T)$ of solutions with all-distinct entries in a subset $T \subseteq \FF_q^n$.
We will additionally need the set 
\begin{equation}
    \mS'_L(T) = \{\bs \in T^m : L(\bs) =  {\bf 0} \} \supseteq \mS_L(T).
\end{equation}
of all solutions with not necessarily distinct entries. The following lemma establishes that asymptotically it makes no difference whether one counts elements in $\mS_L(\FF_q^n)$ or in $\mS'_L(\FF_q^n)$.
\begin{proposition} \label{lemma:countsol}
    For any linear map $L$, we have $|\mS'_L(\FF_q^n)| = |\mS_L(\FF_q^n)| \, \big(1 + o(1) \big)$
\end{proposition}
\begin{proof}
    Recall that the matrix $A \in \mM^{m \times r}$ associated with a linear map $L: \FF_q^m \to \FF_q^r$ is assumed to be of full rank, so that $|\mS'_L(T)| = \Theta(q^{n(m-r)})$. Each possible set of repetitions in the $m$ coordinates defines its own linear map that, with the exception of the repetition-free case, has strictly less than $m-r$ degrees of freedom and therefore $o(q^{n(m-r)})$ solutions. Since there is only a finite number of possible repetitions, asymptotically almost every element from $\mS'_L(\FF_q^n)$ also lies in $\mS_L(\FF_q^n)$.
\end{proof}
In order to develop the flag algebra approach, our notion of solutions will also need to fulfill an averaging equality, i.e., the density of solutions needs to be representable as the weighted density of particular flags, and it needs to be invariant under an appropriate notion of isomorphism.
This is unfortunately not true for $\mS_L(\FF_q^n)$ or $\mS'_L(\FF_q^n)$, so for any $t \geq -1$ and $n \geq t^+$ we introduce the \emphdef{$t$-fixed dimension} $\dim_t(s)$ of a solution $\bs \in \mS'_L(\FF_q^n)$ as the smallest $k \geq t^+$ for which there exists a $t$-fixed $k$-dimensional subspace of $\FF_q^n$ containing all entries of $\bs$. Note that for our purposes, we will only need the unfixed and $0$-fixed dimension, but the notions are easier to state in broader generality. We denote by
\begin{equation}
    \dim_t(L) = \max \{\dim_t(\bs) : \bs \in \mS_L'(\FF_q^n), n \geq t^+\}
\end{equation}
the largest $t$-fixed dimension of any solution to a given linear map $L$. Let us also say that a given linear map $L$ is \emphdef{invariant} if for any solution $\bs = (x_1, \ldots, x_m) \in \mS'_L(\FF_q^n)$ and element $a \in \FF_q^n$ we have $a + \bs = (a+x_1, \ldots, a+x_m) \in \mS'_L(\FF_q^n)$.\footnote{Note that APs clearly are invariant while Schur triples are not. In general a linear map will be invariant if and only if the columns of its associated matrix $A$ sum up to $\bf 0$ modulo $q$.} 
\begin{lemma} \label{lemma:dimbound}
    We have $\dim_t(L) = m-r+t$ for any linear map $L$ when $t \geq 0$ as well as $\dim_{-1}(L) = m-r-1$ when $L$ is invariant.
\end{lemma}
\begin{proof}
    Recall that the matrix $A \in \mM^{m \times r}$ associated with a linear map $L: \FF_q^m \to \FF_q^r$ is assumed to be of full rank, so that we have $m-r$ degrees of freedom. When $t \geq 0$, these can be chosen to lie outside of the fixed $t$-dimensional subspace and together with it form a $t$-fixed $m-r+t$-dimensional subspace. When $t = -1$, the freedom given by the affine term reduces the degree of freedom by exactly $1$ for invariant $L$, as $\dim_{-1}(\bs) = \dim_{-1}(\bs - a)$ for any $a \in \FF_q^n$ and $\bs - a \in \mS'_L(\FF_q^n)$ by assumption.
\end{proof}
We therefore say that $L$ is \emphdef{admissible} if $t \geq 0$ or if $t=-1$ and $L$ is invariant. A solution $\bs \in \mS_L(\FF_q^n)$ for some admissible $L$ is \emphdef{$t$-fixed fully dimensional} if $\dim_t(s)$ attains the respective upper bound stated in \cref{lemma:dimbound}. 
For a given set $T \subseteq \FF_q^n$, we denote the set of fully dimensional solutions to some admissible $L$ by 
\begin{equation}
    \mS_L^t(T) = \{\bs \in \mS'_L(T) : \dim_t(\bs) = \dim_t(L) \}
\end{equation}
and also write $s_L^t(T) = |\mS_L^t(T)| \, / \, |\mS_L^t(\FF_q^n)| $. The important property that we will make use of is that each fully-dimensional solution defines a unique $\dim(L)$-dimensional $t$-fixed subspace in which it lies. Using this, let us establish that asymptotically it makes no difference whether we are counting elements in $\mS_L(\FF_q^n)$ or in $\mS^t_L(\FF_q^n)$ when $n$ tends to infinity.
\begin{proposition}\label{lemma:countfdsol}
    For any $t \geq -1$ and admissible linear map $L$, we have
    \begin{equation*}
        |\mS_L^t(\FF_q^n)| = |\mS(\FF_q^n)| \, \big(1 + o(1) \big).
    \end{equation*}
\end{proposition}
\begin{proof}
    Let us write $d = \dim(L)$ and let $\mS_L^{t,k}(\FF_q^n) = \{ \bs \in \mS_L'(\FF_q^n) : \dim(\bs) = k \}$ denote the set of $k$-dimensional $t$-fixed solutions in $\FF_q^n$ for any $t^+ \leq k \leq d$. Since every solution $\bs \in \mS_L^{t,k} (\FF_q^n)$ defines a unique $k$-dimensional subspace of $\FF_q^n$ that fully contains its entries, it follows that $|\mS_L^{t,k} (\FF_q^n)| = |\Mon_t(k;n)| |\mS_L^{t,k} (\FF_q^k)|$. Using \cref{lemma:count_morphisms} as well as $|\mS_L^{t,k} (\FF_q^k)|  = \theta(1)$ and $\binom{n}{k}_q = \Theta(q^{kn})$ as $n$ tends to infinity, we therefore get
    \begin{equation*}
        |\mS_L^{t,k} (\FF_q^n)| = q^{n-k} \, \binom{n}{k}_q \, |\mS_L^{t,k} (\FF_q^k)| = \Theta(q^{(k-t) \, n}).
    \end{equation*}
    for any $t \geq -1$. It follows that $|\mS_L^{t,k} (\FF_q^n)| = o(|\mS_L^d (\FF_q^n)|)$ when $k < d$, and since $\mS_L^{t,d}(\FF_q^n) = \mS_L^t (\FF_q^n)$, it follows that $|\mS'_L(\FF_q^n)| = |\mS^t(\FF_q^n)| \, \big(1 + o(1) \big)$. The desired statement therefore therefore follows by \cref{lemma:countsol}.
\end{proof}
The following remark motivates why we consider both unfixed and $0$-fixed morphisms.
\begin{remark} \label{rem:solinvariance}
    For any $t \geq -1$, admissible linear map $L$, integer $n \geq \dim_t(L)$, $t$-fixed isomorphism $\varphi: \FF_q^n \to \FF_q^n$, and solution $(x_1, \ldots, x_m) \in \mS^t_L(\FF_q^n)$, we have $\big( \varphi(x_1), \ldots, \varphi(x_m) \big) \in \mS^t_L(\FF_q^n)$, i.e., the number of solutions in a subset of $\FF_q^n$ is invariant under $t$-fixed isomorphism. 
\end{remark}
The same would not hold for $t = -1$ if $L$ was not invariant. In general we therefore need to consider $0$-fixed morphisms, but whenever exclusively dealing with invariant structures, we can be more economical, as unfixed morphisms lead to a smaller number of isomorphism classes of colorings.
Finally, let us conclude this section by showing that fully-dimensional solutions satisfy the desired averaging property.
\begin{proposition} \label{lemma:solutiondoublecounting}
    Let $t \geq -1$ as well as an admissible linear map $L$ be given. For any integers $\dim_t(L) \leq k \leq n$, coloring $\gamma \in \Gamma^t(n)$, and color $1 \leq i \leq c$ , we have 
    \begin{equation*} \label{eq:solutiondoublecounting}
        s_L^t(\gamma^{(i)}) = \sum_{\delta \in \Gamma^t(k)} s_L^t(\delta^{(i)}) \, p_t (\delta; \gamma).
    \end{equation*}
\end{proposition}
\begin{proof}
    To shorten notation, let us write $d = \dim_t(L)$ as well as $\varphi^{-1}(T) = \{ x \in \FF_q^{d} : \varphi(x) \in T \}$ for any $\varphi \in \Mon_t(d;n)$. Since every $t$-fixed fully-dimensional solution defines a unique $d$-dimensional $t$-fixed subspace in which its entries lie, we have
    \begin{equation*}
        |\mS_L^t(\gamma^{(i)})| = \sum_{\varphi \in \Mon_t(d; n)} |\mS_L^t \big( (\gamma \circ \varphi)^{(i)} \big)| = \sum_{\delta \in \Gamma^t(d)} |\mS_L^t (\delta^{(i)})| \cdot |\{ \varphi \in \Mon_t(d;n) : \gamma \circ \varphi \cong_t \delta \}|.
    \end{equation*}
    In particular, this implies $|\mS_L^t(\FF_q^n)| = |\mS_L^t (\FF_q^d)| \, |\Mon_t(d;n)|$, so dividing both sides by $|S_L^t(\FF_q^n)|$ establishes the statement when $k = d$. \cref{lemma:colordoublecounting} extends it to any $k > d$.
\end{proof}

\section{Flag Algebras for Vector Spaces over Finite Fields} \label{sec:flagalgebras}

The theory of flag algebras is due to Razborov~\cite{razborov2007flag} and covers a particular subset of problems in extremal combinatorics dealing with the relation of asymptotic densities. Among other approaches, it allows one to derive lower bounds for the type of problem studied in this paper by formulating and solving a semi-definite program. This has been successfully applied to many longstanding open problems in extremal graph theory~\cite{razborov20103, FalgasVaughan_2013,grzesik2012maximum,baber2011hypergraphs,balogh2022spectrum,CummingsEtAl_2013,balogh2017rainbow} and there are several good introductions to the topic, for example by Silva et al.~\cite{SilvaEtAl_2016}. The purpose of this section is to present a succinct and self-contained version of this framework for our particular application along with some problem-specific nuances.

\subsection{Admissible parameters for flag algebras}\label{sec:admissible_parameters}

We assume that we are given a parameter $\lambda: \Gamma \to \RR$ for which there exists $t_\lambda \in \{-1, 0\}$ such that $\lambda$ is invariant under $t_\lambda$-fixed isomorphisms, meaning that $\lambda (\gamma) = \lambda(\gamma')$ whenever $\gamma \cong_{t_\lambda} \gamma'$, as well as an integer $n_\lambda \in \NN$ such that 
\begin{equation} \label{eq:avgrequirement}
    \lambda (\gamma) = \sum_{\beta \in \Gamma^{t_\lambda}(n)} \lambda(\beta) \, p_{t_\lambda}(\beta, \gamma)
\end{equation}
for all $\gamma \in \Gamma^{t_\lambda}$ and $n_\lambda \leq n \leq \dim(\gamma)$. Note that we established through \cref{rem:solinvariance} and \cref{lemma:solutiondoublecounting} in the previous section that the fraction of monochromatic fully-dimensional solutions to a given linear map $L$ define such a parameter with $t_\lambda = 0$ for general $L$ and $t_\lambda = -1$ for invariant ones, where in either case $n_\lambda = \dim_{t_\lambda}(L)$. Given $\lambda$, we are now interested in determining
\begin{equation} \label{eq:optimization_goal}
    \lambda^\star = \lim_{n \to \infty} \min_{\gamma \in \Gamma^{t_\lambda}(n)} \lambda (\gamma).
\end{equation}
The limit exists as \cref{eq:avgrequirement} implies monotonicity and it is easy to see that one has a trivial lower bound of $\lambda^\star \geq \min_{\delta \in \Gamma^{t_\lambda}(N)} \lambda (\delta)$ for any $N \geq n_\lambda$.

Before proceeding with a formal definition of flag algebras, let us note that everything introduced in this section can be easily extended to cover the case of colorings avoiding a finite set $\mX \subset \Gamma$ of forbidden colorings by simply considering the set $\Gamma_\mX$ colorings not containing any element of $\mX$ as a non-degenerate induced sub-coloring instead of $\Gamma$ throughout. This allows one to address Turán-type questions or even analogue problems to studying the maximum number of cliques in graphs with bounded independence number~\cite{Erdos_1962, Nikiforov_2001, DasEtAl_2013, PikhurkoVaughan_2013}. Since the problems stated in the introduction exclusively deal with the case of $\mX = \emptyset$, we do not further emphasize this point.

\subsection{A formal definition of flag algebras and their semantic cones}

For any $t \geq 0$, we refer to elements of $\Gamma^{t}(t) =  \Gamma(t)$ as \emphdef{types} of dimension $t$. We also introduce a unique \emphdef{empty type}, denoted by $\varnothing$, of dimension $t = -1$. For a given type $\tau$ of dimension $t$ and $n \geq t^+$, we refer to a coloring $F \in \Gamma(n)$ satisfying $F \circ \operatorname{id}_{t,n} \equiv \tau$ as a \emphdef{flag of type $\tau$} (or a \emphdef{$\tau$-flag}), where we recall that $\operatorname{id}_{t,n}$ denotes the unique $t$-fixed isomorphism from $\FF_q^t$ to $\FF_q^n$ and also note that the requirement is vacantly true for $t=-1$. We will write
\begin{equation*}
    \mF^\tau_n = \{ F \in \Gamma(n): F \circ \operatorname{id}_{t,n} \equiv \tau \} / \cong_t
\end{equation*}
for the set of all flags of given type $\tau$ and dimension $n$ up to $t$-fixed isomorphism, i.e., we again choose a canonical representative for each class. We will also write $\mF^\tau = \bigcup_{n \geq t^+} \mF^\tau$ for all $\tau$-flags of any dimension.
Using these notions, we can now give a formal definition of flag algebras. To stick to common notation in flag algebra literature, we will use $F, H, G$ rather than $\delta, \beta, \gamma$ to denote flags from this point on.

\begin{definition}
    For a given type $\tau$, its \emphdef{flag algebra} $\mA^\tau$ is given by considering $\RR \mF^\tau / \mK^\tau$, where 
    \begin{equation}
        \mK^\tau = \Big\{ F - \sum_{F' \in \mF_n^\tau } p_t (F;F') \, F' : F \in \mF^\tau, n \geq \dim(F) \Big\},
    \end{equation}
    and equipping it with the product given by the the bilinear extension of 
    \begin{equation}
        F_1 \cdot F_2 = \sum_{H \in \mF_n^\tau} p_t (F_1, F_2; H) \, H + \mK^\tau
    \end{equation}
    defined for any two flags $F_1, F_2 \in \mF^\tau$ and arbitrary $n \geq \dim(F_1) + \dim(F_2) - \dim (\tau)$.
\end{definition}

Note that the choice of $n$ in the definition of the product is indeed arbitrary by \cref{lemma:colordoublecounting} and that $\tau$ is the multiplicative unit element of $\mA^\tau$. We will omit $\tau$ along with $\mK^\tau$ from any algebraic expressions involving elements of $\mA^\tau$, so given a constant $c \in \RR$ we for example simply write $c$ for the element $c \, \tau + \mK^\tau$ in $\mA^\tau$. Elements of the algebra not directly identified with some flag will commonly be denoted by a lower case $f$.

From a combinatorial perspective, we care about \emphdef{convergent} sequences $(G_k)_{k \in \NN}$ in $\mF^\tau$, that is sequences of flags $G_k \in \mF^\tau$ where $\dim(G_k)$ tends to infinity and for which $\lim_{k \to \infty} p_t(F; G_k)$ exists for any $F \in \mF^\tau$. Note that by compactness every increasing sequence has a convergent subsequence~\cite[Theorem 3.2]{razborov2007flag}. The decisive property of flag algebras is that \emphdef{positive homomorphisms} $\phi$, that is algebra homomorphisms $\phi \in \operatorname{Hom}(\mA^\tau, \RR)$ satisfying $\phi(F) \geq 0$ for all $F \in \mF^\tau$, are in a one-to-one correspondence with \emphdef{limit functionals}, that is the linear extension of maps $f: \mF^\tau \to \RR$ defined through $f(F) = \lim_{k \to \infty} p(F; G_k)$ for some convergent sequence $(G_k)_{k \in \NN}$.

\begin{theorem}[Theorem 3.3 in~\cite{razborov2007flag}]
	Every limit functional is a positive homomorphism and every positive homomorphism is a limit functional.
\end{theorem}

Using this insight, let us write $\operatorname{Hom}^+(\mA^\tau, \RR)$ for the set of positive homomorphisms and
\begin{equation}
	\mS^\tau = \{f \in \mA^\tau: \phi(f) \geq 0 \text{ for all } \phi \in \operatorname{Hom}^+(\mA^\tau, \RR) \}
\end{equation}
for the \emphdef{semantic cone} of type $\tau$. Since the $\phi$ are algebra homomorphisms, we have $c \cdot f \in \mS^\tau$ as well as $f_1 + f_2 \in \mS^\tau$ for any $f, f_1, f_2 \in \mS^\tau$ and $c \in \RR_{\geq 0}$, that is $\mS^\tau$ is indeed a cone for any type $\tau$. We use the semantic cone $\mS^\tau$ to define a partial order in $\mA^\tau$, where $f_1 \geq f_2$ whenever $f_1 - f_2 \in \mS^\tau$. Writing
\begin{equation}
	C^\tau_\lambda = \sum_{\beta \in \mF^\tau_{n_\lambda}} \lambda(\beta) \, \beta,
\end{equation}
for any type $\tau$ of dimension $t_\lambda$ with $t_\lambda$ and $n_\lambda$ as stated at the beginning of \cref{sec:admissible_parameters}, our problem of determining $\lambda^\star$ in \cref{eq:optimization_goal} can now be restated through the conic optimization problem
\begin{equation} \label{eq:optimization_goal_cone}
	\lambda^\star = \max \{ \lambda' \in \RR: C^\tau_\lambda \geq \lambda' \text{ for all types } \tau \text{ of dimension }t_\lambda\}.
\end{equation}

As is the case for graphs, it is often more natural to construct an argument for some type $\tau$ not directly used in \cref{eq:optimization_goal_cone}. We will need to {\lq}pull{\rq} these arguments into the space of flags that are of interest to us. Given a type $\tau$ of dimension $t \geq t_\lambda$, we let $\tau_\lambda$ denote the unique type of dimension $0$ satisfying $\tau_\lambda(0) = \tau(0)$ if $t_\lambda = 0$ and the empty type $\varnothing$ if $t_\lambda = -1$. For a given flag $F \in \mF^\tau_n$, we now let $F \!\! \mid_{t_\lambda}$ denote the element in $\mF^{\tau_\lambda}_n$ obtained by considering $F$ as an element of $\Gamma^{t_\lambda}$, that is we forget about all fixed parts for $t_\lambda = -1$ and all but the fixed zero for $t_\lambda = 0$.
\begin{definition}
    Let a type $\tau$ of dimension $t \geq t_\lambda$ be given. The \emphdef{downward operator} $\llbracket \,\cdot\, \rrbracket_{t_\lambda} \!\!: \mA^{\tau} \to \mA^{\tau_\lambda}$ is given by the linear extension of $F \mapsto q_{t_\lambda} (F) \cdot F \!\! \mid_{t_\lambda}$ defined for any $F \in \mF^{\tau}_n$, where $q_{\lambda} (F)$ denotes the probability that a $t_\lambda$-fixed isomorphism of $\FF_q^n$ chosen uniformly at random turns $F\!\! \mid_{t_\lambda}$ into an element of $\mF^\tau_n$.
\end{definition}
The decisive property of the downward operator, besides being linear~\cite[Theorem 2.5]{razborov2007flag}, is that it is \emphdef{order preserving}, i.e., that $f \geq 0$ in $\mA^{\tau}$ implies $\llbracket f \rrbracket_{t_\lambda} \geq 0$ in $\mA^{\tau_\lambda}$, or equivalently $\llbracket \mS^{\tau} \rrbracket_{t_\lambda} \subseteq \mS^{\tau_\lambda}$~\cite[Theorem 3.1]{razborov2007flag}. Note that this is closely linked to the Cauchy-Schwarz inequality, which in the context of flag algebras can bee seen as stating $\llbracket f^2 \rrbracket_{t_\lambda} \geq \llbracket f \rrbracket_{t_\lambda}^2$~\cite[Theorem 3.14]{razborov2007flag}.

\subsection{How to obtain lower bounds through flag algebras}

One way of establishing that $C_\lambda^{\tau} \geq \lambda'$ for a specific type $\tau$ of dimension $t_\lambda \in \{-1,0\}$, is to find positive elements $f_1, \ldots, f_m \in \mS^{\tau}$ for which one can show that $C_\lambda^{\tau} - \lambda' \geq f_1 + \ldots + f_m$ holds in $\mA^\tau$. Not only does that establish a lower bound, but it can also tell us structural information about extremal convergent sequences, since we must have $\phi(f_i) = 0$ for any $1 \leq i \leq m$ and $\phi \in \operatorname{Hom}^+(\mA^\tau, \RR)$ satisfying $\phi(C_\lambda^{\tau}) = \lambda'$. One obvious candidate for positive elements are squares, so we can establish a lower bound by defining a set of types $\mT$ that satisfy $\tau'_\lambda = \tau$ for any $\tau' \in \mT$  as well as finite sets of algebra elements $\mB_{\tau'} \subset \mA^{\tau'}$, and verifying that
\begin{equation} \label{eq:sos-lower-bound}
	C_\lambda^{\tau} - \lambda' \geq  \sum_{\tau' \in \mT} \sum_{f \in \mB_{\tau'}} \llbracket f^2 \rrbracket_{t_\lambda}.
\end{equation}
Note that for $t_\lambda = -1$ we only need to establish this once for the empty type $\tau = \varnothing$, while for $t_\lambda = 0$ we in fact need to verify this independently for all $c$ types $\tau$ of dimension $0$. However, when the parameter at hand is invariant under permutations of the colors, as is the case for the results stated in the introduction, it is in fact sufficient to verify the above for one arbitrary but fixed type of dimension $0$.

In practice, \cref{eq:sos-lower-bound} will be verifiable over $\mF^\tau_N$ for some fixed (and reasonably small) $N \geq 0$. The types $\mT$ will have dimension $-1 \leq t \leq N-2$ and for each given type $\tau'$ of dimension $t$ the set $\mB_{\tau'}$ will consist of the cosets of linear combinations of flags in $\mF^{\tau'}_{\lfloor(N-t)/2\rfloor}$ if $t \geq 0$ and $\mF^{\tau'}_{N-1}$ when $t = -1$. Such expressions can be found using a well established connection between sum-of-squares (SOS) and semi-definite programming (SDP)~\cite{parrilo2003minimizing}. We will not emphasize these technical aspects as they have been explored already in other works.

\subsection{Explicit constructions for upper bounds} \label{sec:blowup}

Let us finally take a brief detour from flag algebras and explore how to obtain upper bounds for a parameter $\lambda$ satisfying the requirements specified in \cref{sec:admissible_parameters}. While different ways of obtaining upper bounds have been used in the past, by far the most powerful method seems to be an analogue of the simple graph blow-up.

\pagebreak

For any $n \geq 1$, $x = (x_1, \ldots, x_n) \in \FF_q^n$ and $S = \{i_1 < \ldots < i_k\} \subseteq [n]$, we let $x\!\!\restriction_S = (x_{i_1}, \ldots, x_{i_k}) \in \FF_q^k$ denote the restriction of $x$ to its entries with indices in $S$. When $k = 0$ and therefore $S = \emptyset$, we let $x\!\!\restriction_\emptyset = 0 \in \FF_q^0$. For $k \geq 0$ the $k$-fold \emphdef{blow-up} $\gamma^{[k]} \in  \Gamma^t(n + k)$ of a given coloring $\gamma \in \Gamma^t(n)$ with $t \geq -1$ and $n \geq t^+$ is now defined by
\begin{equation*}
    \gamma^{[k]}(x) = \gamma(x\!\!\restriction_n)
\end{equation*}
for any $x \in \FF_q^{n + k}$.
The decisive property of the blow-up is its invariance with respect to the degenerate densities of subcolorings. Let us introduce a degenerate version of the parameter $\lambda$ by writing 
\begin{equation}
     \lambda^d(\gamma) = \sum_{\beta \in \Gamma^{t_\lambda}(n_\lambda)} \lambda(\beta) \, p_{t_\lambda}^d (\beta, \gamma)
\end{equation}
for any coloring $\gamma \in \Gamma^{t_\lambda}$. The following lemma states that we can derive an upper bound for $\lambda^\star$ from any coloring.
\begin{lemma} \label{lemma:blow_ub}
    We have $\lambda^\star \leq \lambda^d(\gamma)$ for any coloring $\gamma \in \Gamma^{t_\lambda}$.
\end{lemma}
\begin{proof}
    Clearly $\lambda^\star \leq \liminf_{k \to \infty} \lambda(\gamma^{[k]})$ by definition of $\lambda^\star$. By \cref{lemma:asymptoticdensity}, we have $\lambda(\gamma^{[k]}) = \lambda^d(\gamma^{[k]}) (1 + o_k(1))$. Given any coloring $\delta \in \Gamma^t(k')$ with $k' \geq 0$, we have $p_{t_\lambda}^d(\delta, \gamma^{[k]}) = p_{t_\lambda}^d(\delta, \gamma)$ and therefore $\lambda^d(\gamma^{[k]}) = \lambda^d(\gamma)$, establishing the lemma.
\end{proof}
This basic argument will already allow us to establish tight bounds in some cases. When certain requriments are met, one can however iterate a given construction to obtain an improved upper bound. For a given $\gamma \in \Gamma^{t}(n)$ and $\beta \in \Gamma(k)$ with $t \geq -1$, $n \geq t^+$, and $k \geq 0$, the \emphdef{product construction} $\gamma \otimes \beta \in \Gamma^t(n + k)$ is given by
\begin{equation*}
    \gamma \otimes \beta (x) = \begin{cases}
    \beta(x \!\!\restriction_{\{n+1, \ldots n+k\}}) & \text{if } x \!\!\restriction_{\{1,\ldots,n\}} = 0, \\
    \gamma(x \!\!\restriction_{\{1,\ldots,n\}}) & \text{otherwise.}
    \end{cases}
\end{equation*}
We note that a similar notion to the product construction has previously been used in the additive setting when dealing with structures in $\ZZ_n$, see~\cite{LuPeng_2012, parczyk2022new}. In order to state the next proposition, let $\gamma_i = i \cdot \mathds{1}_0 +  (1-\mathds{1}_0) \cdot \gamma$ denote the coloring obtained from $\gamma$ for any $1 \leq i \leq c$ by setting $\gamma_i(0) = i$ and $\gamma_0 \in \Gamma(0)$ the coloring given by $\gamma_0(0) = \gamma(0)$.
\begin{proposition} \label{prop:it_ub}
    For any $\gamma \in \Gamma^{t_\lambda}(n)$ satisfying $\lambda^d(\gamma_i) = \lambda^d(\gamma)$ for all $1 \leq i \leq c$, we have
    \begin{equation} \label{eq:it_ub}
        \lambda^\star \leq \frac{|\M_{t_\lambda}(n_{\lambda}, n)| \, \lambda^d(\gamma) - \lambda^d(\gamma_0)}{|\M_{t_\lambda}(n_{\lambda}, n)| - 1}.
    \end{equation}
\end{proposition}
\begin{proof}
    Since $\lambda^d(\gamma_i) = \lambda^d(\gamma)$ for all $1 \leq i \leq c$, we have 
    \begin{equation*}
        \lambda^d(\gamma \otimes \gamma) = \lambda^d(\gamma)  - \frac{\lambda^d(\gamma_0) - \lambda^d(\gamma)}{|\M_{t_\lambda}(n_\lambda, n)|}.
    \end{equation*}
    Iterating the product construction and considering the limit of $\lambda^d(\gamma \otimes \gamma \otimes \ldots \otimes \gamma)$ gives the desired statement.
\end{proof}

\section{Proofs of the Ramsey Multiplicity Statements} \label{sec:proofs}

All upper bounds rely either on \cref{lemma:blow_ub} or \cref{prop:it_ub} and the respective colorings were generally found computationally, either by checking all colorings up to a certain dimension exhaustively or through search heuristics as recently explored by Parczyk et al.~\cite{parczyk2022new} in the context of graph theory.  Certificates for lower bounds will be stated using the flag algebra framework. The coefficients were found by first numerically solving an SDP formulation using {\tt csdp}~\cite{csdp} and then either rounding an LDL-decomposition of the output to exact rational values or by solving an exact rational Linear Program (LP) using {\tt SoPlex}~\cite{GamrathAndersonBestuzhevaetal.2020, GleixnerSteffyWolter2015} whenever an exact value for $\lambda'$ needed to by met. In the latter case, the zero eigenvector space of the solution could in general be easily guessed from the numerical solution and enforced in the LP formulation. We will represent the certificate in its SOS form as in \cref{eq:sos-lower-bound}. We made an effort to represent the proofs in a self-contained way, but the code needed to verify some of the statements here can be found online at \url{github.com/FordUniver/rs_radomult_23}.

\subsection{Proof of \cref{prop:4apresult}}

\paragraph{Upper bound.} The value of $13/126$ is obtained from an iterative blow-up, that is by applying \cref{prop:it_ub} to the following $3$-dimensional $2$-coloring:
\medskip
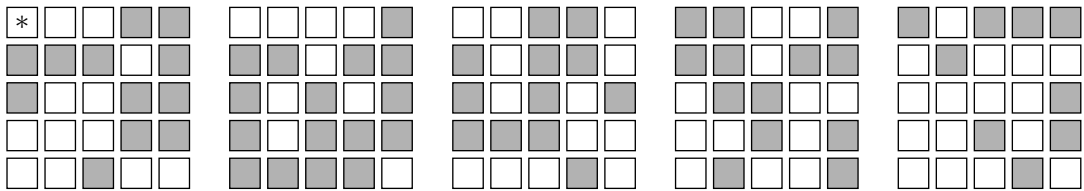
\begin{center}

\begin{tikzpicture}
    \draw[line width=0.5pt, fill=white]    (0*\squaresize + 0*\padding, 0*\squaresize + 0*\padding) rectangle ++(\squaresize,\squaresize);
    \draw[line width=0.5pt, fill=white]    (1*\squaresize + 1*\padding, 0*\squaresize + 0*\padding) rectangle ++(\squaresize,\squaresize);
    \draw[line width=0.5pt, fill=black!30] (2*\squaresize + 2*\padding, 0*\squaresize + 0*\padding) rectangle ++(\squaresize,\squaresize);
    \draw[line width=0.5pt, fill=white]    (3*\squaresize + 3*\padding, 0*\squaresize + 0*\padding) rectangle ++(\squaresize,\squaresize);
    \draw[line width=0.5pt, fill=white] (4*\squaresize + 4*\padding, 0*\squaresize + 0*\padding) rectangle ++(\squaresize,\squaresize);
    \draw[line width=0.5pt, fill=white]    (0*\squaresize + 0*\padding, 1*\squaresize + 1*\padding) rectangle ++(\squaresize,\squaresize);
    \draw[line width=0.5pt, fill=white] (1*\squaresize + 1*\padding, 1*\squaresize + 1*\padding) rectangle ++(\squaresize,\squaresize);
    \draw[line width=0.5pt, fill=white]    (2*\squaresize + 2*\padding, 1*\squaresize + 1*\padding) rectangle ++(\squaresize,\squaresize);
    \draw[line width=0.5pt, fill=black!30]    (3*\squaresize + 3*\padding, 1*\squaresize + 1*\padding) rectangle ++(\squaresize,\squaresize);
    \draw[line width=0.5pt, fill=black!30]    (4*\squaresize + 4*\padding, 1*\squaresize + 1*\padding) rectangle ++(\squaresize,\squaresize);
    \draw[line width=0.5pt, fill=black!30] (0*\squaresize + 0*\padding, 2*\squaresize + 2*\padding) rectangle ++(\squaresize,\squaresize);
    \draw[line width=0.5pt, fill=white] (1*\squaresize + 1*\padding, 2*\squaresize + 2*\padding) rectangle ++(\squaresize,\squaresize);
    \draw[line width=0.5pt, fill=white] (2*\squaresize + 2*\padding, 2*\squaresize + 2*\padding) rectangle ++(\squaresize,\squaresize);
    \draw[line width=0.5pt, fill=black!30] (3*\squaresize + 3*\padding, 2*\squaresize + 2*\padding) rectangle ++(\squaresize,\squaresize);
    \draw[line width=0.5pt, fill=black!30]    (4*\squaresize + 4*\padding, 2*\squaresize + 2*\padding) rectangle ++(\squaresize,\squaresize);
    %2
    \draw[line width=0.5pt, fill=black!30] (0*\squaresize + 0*\padding, 3*\squaresize + 3*\padding) rectangle ++(\squaresize,\squaresize);
    \draw[line width=0.5pt, fill=black!30] (1*\squaresize + 1*\padding, 3*\squaresize + 3*\padding) rectangle ++(\squaresize,\squaresize);
    \draw[line width=0.5pt, fill=black!30] (2*\squaresize + 2*\padding, 3*\squaresize + 3*\padding) rectangle ++(\squaresize,\squaresize);
    \draw[line width=0.5pt, fill=white] (3*\squaresize + 3*\padding, 3*\squaresize + 3*\padding) rectangle ++(\squaresize,\squaresize);
    \draw[line width=0.5pt, fill=black!30]    (4*\squaresize + 4*\padding, 3*\squaresize + 3*\padding) rectangle ++(\squaresize,\squaresize);
    \draw[line width=0.5pt, fill=white] (0*\squaresize + 0*\padding, 4*\squaresize + 4*\padding) rectangle ++(\squaresize,\squaresize);
    \node at (0.5*\squaresize, 5.5*\squaresize) {$\ast$}; % pattern=north west lines, pattern color=black!30
    \draw[line width=0.5pt, fill=white]    (1*\squaresize + 1*\padding, 4*\squaresize + 4*\padding) rectangle ++(\squaresize,\squaresize);
    \draw[line width=0.5pt, fill=white]    (2*\squaresize + 2*\padding, 4*\squaresize + 4*\padding) rectangle ++(\squaresize,\squaresize);
    \draw[line width=0.5pt, fill=black!30]    (3*\squaresize + 3*\padding, 4*\squaresize + 4*\padding) rectangle ++(\squaresize,\squaresize);
    \draw[line width=0.5pt, fill=black!30]    (4*\squaresize + 4*\padding, 4*\squaresize + 4*\padding) rectangle ++(\squaresize,\squaresize);
\end{tikzpicture}
\quad
\begin{tikzpicture}
    \draw[line width=0.5pt, fill=black!30]    (0*\squaresize + 0*\padding, 0*\squaresize + 0*\padding) rectangle ++(\squaresize,\squaresize);
    \draw[line width=0.5pt, fill=black!30]    (1*\squaresize + 1*\padding, 0*\squaresize + 0*\padding) rectangle ++(\squaresize,\squaresize);
    \draw[line width=0.5pt, fill=black!30] (2*\squaresize + 2*\padding, 0*\squaresize + 0*\padding) rectangle ++(\squaresize,\squaresize);
    \draw[line width=0.5pt, fill=black!30]    (3*\squaresize + 3*\padding, 0*\squaresize + 0*\padding) rectangle ++(\squaresize,\squaresize);
    \draw[line width=0.5pt, fill=white] (4*\squaresize + 4*\padding, 0*\squaresize + 0*\padding) rectangle ++(\squaresize,\squaresize);
    \draw[line width=0.5pt, fill=black!30]    (0*\squaresize + 0*\padding, 1*\squaresize + 1*\padding) rectangle ++(\squaresize,\squaresize);
    \draw[line width=0.5pt, fill=white] (1*\squaresize + 1*\padding, 1*\squaresize + 1*\padding) rectangle ++(\squaresize,\squaresize);
    \draw[line width=0.5pt, fill=black!30]    (2*\squaresize + 2*\padding, 1*\squaresize + 1*\padding) rectangle ++(\squaresize,\squaresize);
    \draw[line width=0.5pt, fill=black!30]    (3*\squaresize + 3*\padding, 1*\squaresize + 1*\padding) rectangle ++(\squaresize,\squaresize);
    \draw[line width=0.5pt, fill=black!30]    (4*\squaresize + 4*\padding, 1*\squaresize + 1*\padding) rectangle ++(\squaresize,\squaresize);
    \draw[line width=0.5pt, fill=black!30] (0*\squaresize + 0*\padding, 2*\squaresize + 2*\padding) rectangle ++(\squaresize,\squaresize);
    \draw[line width=0.5pt, fill=white] (1*\squaresize + 1*\padding, 2*\squaresize + 2*\padding) rectangle ++(\squaresize,\squaresize);
    \draw[line width=0.5pt, fill=black!30] (2*\squaresize + 2*\padding, 2*\squaresize + 2*\padding) rectangle ++(\squaresize,\squaresize);
    \draw[line width=0.5pt, fill=white] (3*\squaresize + 3*\padding, 2*\squaresize + 2*\padding) rectangle ++(\squaresize,\squaresize);
    \draw[line width=0.5pt, fill=black!30]    (4*\squaresize + 4*\padding, 2*\squaresize + 2*\padding) rectangle ++(\squaresize,\squaresize);
    %2
    \draw[line width=0.5pt, fill=black!30] (0*\squaresize + 0*\padding, 3*\squaresize + 3*\padding) rectangle ++(\squaresize,\squaresize);
    \draw[line width=0.5pt, fill=black!30] (1*\squaresize + 1*\padding, 3*\squaresize + 3*\padding) rectangle ++(\squaresize,\squaresize);
    \draw[line width=0.5pt, fill=white] (2*\squaresize + 2*\padding, 3*\squaresize + 3*\padding) rectangle ++(\squaresize,\squaresize);
    \draw[line width=0.5pt, fill=black!30] (3*\squaresize + 3*\padding, 3*\squaresize + 3*\padding) rectangle ++(\squaresize,\squaresize);
    \draw[line width=0.5pt, fill=black!30]    (4*\squaresize + 4*\padding, 3*\squaresize + 3*\padding) rectangle ++(\squaresize,\squaresize);
    \draw[line width=0.5pt, fill=white] (0*\squaresize + 0*\padding, 4*\squaresize + 4*\padding) rectangle ++(\squaresize,\squaresize);
    \draw[line width=0.5pt, fill=white]    (1*\squaresize + 1*\padding, 4*\squaresize + 4*\padding) rectangle ++(\squaresize,\squaresize);
    \draw[line width=0.5pt, fill=white]    (2*\squaresize + 2*\padding, 4*\squaresize + 4*\padding) rectangle ++(\squaresize,\squaresize);
    \draw[line width=0.5pt, fill=white]    (3*\squaresize + 3*\padding, 4*\squaresize + 4*\padding) rectangle ++(\squaresize,\squaresize);
    \draw[line width=0.5pt, fill=black!30]    (4*\squaresize + 4*\padding, 4*\squaresize + 4*\padding) rectangle ++(\squaresize,\squaresize);
\end{tikzpicture}
\quad
\begin{tikzpicture}
    \draw[line width=0.5pt, fill=white]    (0*\squaresize + 0*\padding, 0*\squaresize + 0*\padding) rectangle ++(\squaresize,\squaresize);
    \draw[line width=0.5pt, fill=white]    (1*\squaresize + 1*\padding, 0*\squaresize + 0*\padding) rectangle ++(\squaresize,\squaresize);
    \draw[line width=0.5pt, fill=white] (2*\squaresize + 2*\padding, 0*\squaresize + 0*\padding) rectangle ++(\squaresize,\squaresize);
    \draw[line width=0.5pt, fill=black!30]    (3*\squaresize + 3*\padding, 0*\squaresize + 0*\padding) rectangle ++(\squaresize,\squaresize);
    \draw[line width=0.5pt, fill=white] (4*\squaresize + 4*\padding, 0*\squaresize + 0*\padding) rectangle ++(\squaresize,\squaresize);
    \draw[line width=0.5pt, fill=black!30]    (0*\squaresize + 0*\padding, 1*\squaresize + 1*\padding) rectangle ++(\squaresize,\squaresize);
    \draw[line width=0.5pt, fill=black!30] (1*\squaresize + 1*\padding, 1*\squaresize + 1*\padding) rectangle ++(\squaresize,\squaresize);
    \draw[line width=0.5pt, fill=black!30]    (2*\squaresize + 2*\padding, 1*\squaresize + 1*\padding) rectangle ++(\squaresize,\squaresize);
    \draw[line width=0.5pt, fill=white]    (3*\squaresize + 3*\padding, 1*\squaresize + 1*\padding) rectangle ++(\squaresize,\squaresize);
    \draw[line width=0.5pt, fill=white]    (4*\squaresize + 4*\padding, 1*\squaresize + 1*\padding) rectangle ++(\squaresize,\squaresize);
    \draw[line width=0.5pt, fill=black!30] (0*\squaresize + 0*\padding, 2*\squaresize + 2*\padding) rectangle ++(\squaresize,\squaresize);
    \draw[line width=0.5pt, fill=white] (1*\squaresize + 1*\padding, 2*\squaresize + 2*\padding) rectangle ++(\squaresize,\squaresize);
    \draw[line width=0.5pt, fill=black!30] (2*\squaresize + 2*\padding, 2*\squaresize + 2*\padding) rectangle ++(\squaresize,\squaresize);
    \draw[line width=0.5pt, fill=white] (3*\squaresize + 3*\padding, 2*\squaresize + 2*\padding) rectangle ++(\squaresize,\squaresize);
    \draw[line width=0.5pt, fill=black!30]    (4*\squaresize + 4*\padding, 2*\squaresize + 2*\padding) rectangle ++(\squaresize,\squaresize);
    %2
    \draw[line width=0.5pt, fill=black!30] (0*\squaresize + 0*\padding, 3*\squaresize + 3*\padding) rectangle ++(\squaresize,\squaresize);
    \draw[line width=0.5pt, fill=white] (1*\squaresize + 1*\padding, 3*\squaresize + 3*\padding) rectangle ++(\squaresize,\squaresize);
    \draw[line width=0.5pt, fill=black!30] (2*\squaresize + 2*\padding, 3*\squaresize + 3*\padding) rectangle ++(\squaresize,\squaresize);
    \draw[line width=0.5pt, fill=black!30] (3*\squaresize + 3*\padding, 3*\squaresize + 3*\padding) rectangle ++(\squaresize,\squaresize);
    \draw[line width=0.5pt, fill=white]    (4*\squaresize + 4*\padding, 3*\squaresize + 3*\padding) rectangle ++(\squaresize,\squaresize);
    \draw[line width=0.5pt, fill=white] (0*\squaresize + 0*\padding, 4*\squaresize + 4*\padding) rectangle ++(\squaresize,\squaresize);
    \draw[line width=0.5pt, fill=white]    (1*\squaresize + 1*\padding, 4*\squaresize + 4*\padding) rectangle ++(\squaresize,\squaresize);
    \draw[line width=0.5pt, fill=black!30]    (2*\squaresize + 2*\padding, 4*\squaresize + 4*\padding) rectangle ++(\squaresize,\squaresize);
    \draw[line width=0.5pt, fill=black!30]    (3*\squaresize + 3*\padding, 4*\squaresize + 4*\padding) rectangle ++(\squaresize,\squaresize);
    \draw[line width=0.5pt, fill=white]    (4*\squaresize + 4*\padding, 4*\squaresize + 4*\padding) rectangle ++(\squaresize,\squaresize);
\end{tikzpicture}
\quad
\begin{tikzpicture}
    \draw[line width=0.5pt, fill=white]    (0*\squaresize + 0*\padding, 0*\squaresize + 0*\padding) rectangle ++(\squaresize,\squaresize);
    \draw[line width=0.5pt, fill=black!30]    (1*\squaresize + 1*\padding, 0*\squaresize + 0*\padding) rectangle ++(\squaresize,\squaresize);
    \draw[line width=0.5pt, fill=white] (2*\squaresize + 2*\padding, 0*\squaresize + 0*\padding) rectangle ++(\squaresize,\squaresize);
    \draw[line width=0.5pt, fill=white]    (3*\squaresize + 3*\padding, 0*\squaresize + 0*\padding) rectangle ++(\squaresize,\squaresize);
    \draw[line width=0.5pt, fill=black!30] (4*\squaresize + 4*\padding, 0*\squaresize + 0*\padding) rectangle ++(\squaresize,\squaresize);
    \draw[line width=0.5pt, fill=white]    (0*\squaresize + 0*\padding, 1*\squaresize + 1*\padding) rectangle ++(\squaresize,\squaresize);
    \draw[line width=0.5pt, fill=white] (1*\squaresize + 1*\padding, 1*\squaresize + 1*\padding) rectangle ++(\squaresize,\squaresize);
    \draw[line width=0.5pt, fill=black!30]    (2*\squaresize + 2*\padding, 1*\squaresize + 1*\padding) rectangle ++(\squaresize,\squaresize);
    \draw[line width=0.5pt, fill=white]    (3*\squaresize + 3*\padding, 1*\squaresize + 1*\padding) rectangle ++(\squaresize,\squaresize);
    \draw[line width=0.5pt, fill=black!30]    (4*\squaresize + 4*\padding, 1*\squaresize + 1*\padding) rectangle ++(\squaresize,\squaresize);
    \draw[line width=0.5pt, fill=white] (0*\squaresize + 0*\padding, 2*\squaresize + 2*\padding) rectangle ++(\squaresize,\squaresize);
    \draw[line width=0.5pt, fill=black!30] (1*\squaresize + 1*\padding, 2*\squaresize + 2*\padding) rectangle ++(\squaresize,\squaresize);
    \draw[line width=0.5pt, fill=black!30] (2*\squaresize + 2*\padding, 2*\squaresize + 2*\padding) rectangle ++(\squaresize,\squaresize);
    \draw[line width=0.5pt, fill=white] (3*\squaresize + 3*\padding, 2*\squaresize + 2*\padding) rectangle ++(\squaresize,\squaresize);
    \draw[line width=0.5pt, fill=white]    (4*\squaresize + 4*\padding, 2*\squaresize + 2*\padding) rectangle ++(\squaresize,\squaresize);
    %2
    \draw[line width=0.5pt, fill=black!30] (0*\squaresize + 0*\padding, 3*\squaresize + 3*\padding) rectangle ++(\squaresize,\squaresize);
    \draw[line width=0.5pt, fill=black!30] (1*\squaresize + 1*\padding, 3*\squaresize + 3*\padding) rectangle ++(\squaresize,\squaresize);
    \draw[line width=0.5pt, fill=white] (2*\squaresize + 2*\padding, 3*\squaresize + 3*\padding) rectangle ++(\squaresize,\squaresize);
    \draw[line width=0.5pt, fill=black!30] (3*\squaresize + 3*\padding, 3*\squaresize + 3*\padding) rectangle ++(\squaresize,\squaresize);
    \draw[line width=0.5pt, fill=black!30]    (4*\squaresize + 4*\padding, 3*\squaresize + 3*\padding) rectangle ++(\squaresize,\squaresize);
    \draw[line width=0.5pt, fill=black!30] (0*\squaresize + 0*\padding, 4*\squaresize + 4*\padding) rectangle ++(\squaresize,\squaresize);
    \draw[line width=0.5pt, fill=black!30]    (1*\squaresize + 1*\padding, 4*\squaresize + 4*\padding) rectangle ++(\squaresize,\squaresize);
    \draw[line width=0.5pt, fill=white]    (2*\squaresize + 2*\padding, 4*\squaresize + 4*\padding) rectangle ++(\squaresize,\squaresize);
    \draw[line width=0.5pt, fill=white]    (3*\squaresize + 3*\padding, 4*\squaresize + 4*\padding) rectangle ++(\squaresize,\squaresize);
    \draw[line width=0.5pt, fill=black!30]    (4*\squaresize + 4*\padding, 4*\squaresize + 4*\padding) rectangle ++(\squaresize,\squaresize);
\end{tikzpicture}
\quad
\begin{tikzpicture}
    \draw[line width=0.5pt, fill=white]    (0*\squaresize + 0*\padding, 0*\squaresize + 0*\padding) rectangle ++(\squaresize,\squaresize);
    \draw[line width=0.5pt, fill=white]    (1*\squaresize + 1*\padding, 0*\squaresize + 0*\padding) rectangle ++(\squaresize,\squaresize);
    \draw[line width=0.5pt, fill=white] (2*\squaresize + 2*\padding, 0*\squaresize + 0*\padding) rectangle ++(\squaresize,\squaresize);
    \draw[line width=0.5pt, fill=black!30]    (3*\squaresize + 3*\padding, 0*\squaresize + 0*\padding) rectangle ++(\squaresize,\squaresize);
    \draw[line width=0.5pt, fill=white] (4*\squaresize + 4*\padding, 0*\squaresize + 0*\padding) rectangle ++(\squaresize,\squaresize);
    \draw[line width=0.5pt, fill=white]    (0*\squaresize + 0*\padding, 1*\squaresize + 1*\padding) rectangle ++(\squaresize,\squaresize);
    \draw[line width=0.5pt, fill=white] (1*\squaresize + 1*\padding, 1*\squaresize + 1*\padding) rectangle ++(\squaresize,\squaresize);
    \draw[line width=0.5pt, fill=black!30]    (2*\squaresize + 2*\padding, 1*\squaresize + 1*\padding) rectangle ++(\squaresize,\squaresize);
    \draw[line width=0.5pt, fill=white]    (3*\squaresize + 3*\padding, 1*\squaresize + 1*\padding) rectangle ++(\squaresize,\squaresize);
    \draw[line width=0.5pt, fill=black!30]    (4*\squaresize + 4*\padding, 1*\squaresize + 1*\padding) rectangle ++(\squaresize,\squaresize);
    \draw[line width=0.5pt, fill=white] (0*\squaresize + 0*\padding, 2*\squaresize + 2*\padding) rectangle ++(\squaresize,\squaresize);
    \draw[line width=0.5pt, fill=white] (1*\squaresize + 1*\padding, 2*\squaresize + 2*\padding) rectangle ++(\squaresize,\squaresize);
    \draw[line width=0.5pt, fill=white] (2*\squaresize + 2*\padding, 2*\squaresize + 2*\padding) rectangle ++(\squaresize,\squaresize);
    \draw[line width=0.5pt, fill=white] (3*\squaresize + 3*\padding, 2*\squaresize + 2*\padding) rectangle ++(\squaresize,\squaresize);
    \draw[line width=0.5pt, fill=black!30]    (4*\squaresize + 4*\padding, 2*\squaresize + 2*\padding) rectangle ++(\squaresize,\squaresize);
    %2
    \draw[line width=0.5pt, fill=white] (0*\squaresize + 0*\padding, 3*\squaresize + 3*\padding) rectangle ++(\squaresize,\squaresize);
    \draw[line width=0.5pt, fill=black!30] (1*\squaresize + 1*\padding, 3*\squaresize + 3*\padding) rectangle ++(\squaresize,\squaresize);
    \draw[line width=0.5pt, fill=white] (2*\squaresize + 2*\padding, 3*\squaresize + 3*\padding) rectangle ++(\squaresize,\squaresize);
    \draw[line width=0.5pt, fill=white] (3*\squaresize + 3*\padding, 3*\squaresize + 3*\padding) rectangle ++(\squaresize,\squaresize);
    \draw[line width=0.5pt, fill=white]    (4*\squaresize + 4*\padding, 3*\squaresize + 3*\padding) rectangle ++(\squaresize,\squaresize);
    \draw[line width=0.5pt, fill=black!30] (0*\squaresize + 0*\padding, 4*\squaresize + 4*\padding) rectangle ++(\squaresize,\squaresize);
    \draw[line width=0.5pt, fill=white]    (1*\squaresize + 1*\padding, 4*\squaresize + 4*\padding) rectangle ++(\squaresize,\squaresize);
    \draw[line width=0.5pt, fill=black!30]    (2*\squaresize + 2*\padding, 4*\squaresize + 4*\padding) rectangle ++(\squaresize,\squaresize);
    \draw[line width=0.5pt, fill=black!30]    (3*\squaresize + 3*\padding, 4*\squaresize + 4*\padding) rectangle ++(\squaresize,\squaresize);
    \draw[line width=0.5pt, fill=black!30]    (4*\squaresize + 4*\padding, 4*\squaresize + 4*\padding) rectangle ++(\squaresize,\squaresize);
  \end{tikzpicture}
\end{center}
\smallskip
Note that we have highlighted the 0 element whose color can be chosen freely with an asterisk symbol. Verifying the requirements of \cref{prop:it_ub} and evaluating \cref{eq:it_ub} is computationally straightforward.

\paragraph{Lower bound.} In order to state a certificate for the lower bound, let us define some variable names for the relevant flags corresponding to types of dimension $-1$ and $0$ in a graphical way:
\medskip
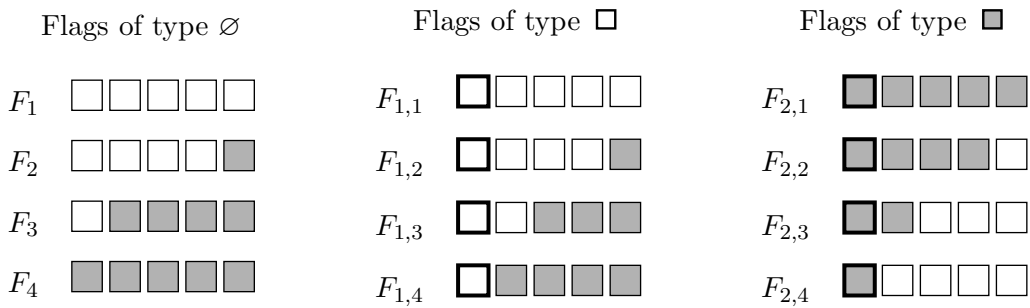
\begin{center}
    \begin{tabular}{cl}
        \multicolumn{2}{c}{Flags of type $\varnothing$} \bigskip \\
        $F_{1}$ & \drawsquares{white, white, white, white, white} \medskip \\
        $F_{2}$ & \drawsquares{white, white, white, white, black!30} \medskip \\
        $F_{3}$ & \drawsquares{white, black!30, black!30, black!30, black!30} \medskip \\
        $F_{4}$ & \drawsquares{black!30, black!30, black!30, black!30, black!30}
    \end{tabular}
    \qquad
    \begin{tabular}{cl}
        \multicolumn{2}{c}{Flags of type \scalebox{0.6}{\drawemphsquares{white}}} \bigskip \\
        $F_{1,1}$ & \drawemphsquares{white, white, white, white, white} \medskip \\
        $F_{1,2}$ & \drawemphsquares{white, white, white, white, black!30} \medskip \\
        $F_{1,3}$ & \drawemphsquares{white, white, black!30, black!30, black!30}\medskip  \\
        $F_{1,4}$ & \drawemphsquares{white, black!30, black!30, black!30, black!30}
    \end{tabular}
    \qquad
    \begin{tabular}{cl}
        \multicolumn{2}{c}{Flags of type \scalebox{0.6}{\drawemphsquares{black!30}}} \bigskip \\
        $F_{2,1}$ & \drawemphsquares{black!30, black!30, black!30, black!30, black!30} \medskip \\
        $F_{2,2}$ & \drawemphsquares{black!30, black!30, black!30, black!30, white} \medskip \\
        $F_{2,3}$ & \drawemphsquares{black!30, black!30, white, white, white} \medskip \\
        $F_{2,4}$ & \drawemphsquares{black!30, white, white, white, white}
    \end{tabular}
\end{center}
\medskip
A proof of the lower bound now follows by first computationally verifying that 
\begin{align*}
    F_1 + F_4 + (F_2 + F_3) / 5 - 1/10 & \geq  \sum_{i=1}^2 \Big( 9/10 \cdot \big\llbracket \big(F_{i,1} + ( 5 \, F_{i,2} - 5 \, F_{i,3} - 10 \, F_{i,4} ) / 27 \big)^2 \big\rrbracket_{-1} \\
    & \phantom{\geq} \quad \ldots + 61/162 \cdot \big\llbracket \big( (F_{i,3} -  F_{i,2}) / 2 + F_{i,4} \big)^2 \big\rrbracket_{-1}  \Big)
\end{align*}
 holds over all elements in $\Gamma_{5,2}^{-1}(2)$, that is all unfixed $2$-colorings of $2$-dimensions. Obtaining all $3324$ such colorings in an efficient way is computationally not entirely trivial but still reasonable with some insights from a practical approach to the graph isomorphism problem~\cite{McKay_1981, McKay_1998, McKayPiperno_2014}. This already implies that $m_{5,2}(L_{4\text{-AP}}) \geq 1/10$. In order to see that this in fact needs to be a strict inequality, let   $\phi$ be a hypothetical limit functional satisfying $\phi(F_1 + F_4 + (F_2 + F_3) / 5) = 1/10$. The summands on the right-hand side imply that $\phi(F_{i,3} -  F_{i,2} + 2 \, F_{i,4}) = 0$ and therefore $\phi(F_{i,1}) = 0$ holds for $i \in \{1,2\}$, which in turn implies $\phi(F_{1}) = \phi(F_{4}) = 0$. This however contradicts the fact that $m_{5,2}(L_{5\text{-AP}}) > 0$, establishing that $m_{5,2}(L_{4\text{-AP}}) > 1/10$.

\subsection{Proof of \cref{thm:3ap3colorresult}}

\paragraph{Upper bound.} The upper bound is obtained from a simple blow-up, that is by applying \cref{lemma:blow_ub} to the following $3$-dimensional $3$-coloring:
\medskip
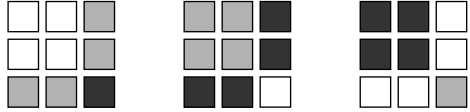
\begin{center}

\begin{tikzpicture}
    \draw[line width=0.5pt, fill=black!30]    (0*\squaresize + 0*\padding, 0*\squaresize + 0*\padding) rectangle ++(\squaresize,\squaresize);
    \draw[line width=0.5pt, fill=black!30]    (1*\squaresize + 1*\padding, 0*\squaresize + 0*\padding) rectangle ++(\squaresize,\squaresize);
    \draw[line width=0.5pt, fill=black!80] (2*\squaresize + 2*\padding, 0*\squaresize + 0*\padding) rectangle ++(\squaresize,\squaresize);
    \draw[line width=0.5pt, fill=white]    (0*\squaresize + 0*\padding, 1*\squaresize + 1*\padding) rectangle ++(\squaresize,\squaresize);
    \draw[line width=0.5pt, fill=white] (1*\squaresize + 1*\padding, 1*\squaresize + 1*\padding) rectangle ++(\squaresize,\squaresize);
    \draw[line width=0.5pt, fill=black!30]    (2*\squaresize + 2*\padding, 1*\squaresize + 1*\padding) rectangle ++(\squaresize,\squaresize);
    \draw[line width=0.5pt, fill=white] (0*\squaresize + 0*\padding, 2*\squaresize + 2*\padding) rectangle ++(\squaresize,\squaresize);
    \draw[line width=0.5pt, fill=white] (1*\squaresize + 1*\padding, 2*\squaresize + 2*\padding) rectangle ++(\squaresize,\squaresize);
    \draw[line width=0.5pt, fill=black!30] (2*\squaresize + 2*\padding, 2*\squaresize + 2*\padding) rectangle ++(\squaresize,\squaresize);
\end{tikzpicture}
\qquad
\begin{tikzpicture}
    \draw[line width=0.5pt, fill=black!80]    (0*\squaresize + 0*\padding, 0*\squaresize + 0*\padding) rectangle ++(\squaresize,\squaresize);
    \draw[line width=0.5pt, fill=black!80]    (1*\squaresize + 1*\padding, 0*\squaresize + 0*\padding) rectangle ++(\squaresize,\squaresize);
    \draw[line width=0.5pt, fill=white] (2*\squaresize + 2*\padding, 0*\squaresize + 0*\padding) rectangle ++(\squaresize,\squaresize);
    \draw[line width=0.5pt, fill=black!30]    (0*\squaresize + 0*\padding, 1*\squaresize + 1*\padding) rectangle ++(\squaresize,\squaresize);
    \draw[line width=0.5pt, fill=black!30] (1*\squaresize + 1*\padding, 1*\squaresize + 1*\padding) rectangle ++(\squaresize,\squaresize);
    \draw[line width=0.5pt, fill=black!80]    (2*\squaresize + 2*\padding, 1*\squaresize + 1*\padding) rectangle ++(\squaresize,\squaresize);
    \draw[line width=0.5pt, fill=black!30] (0*\squaresize + 0*\padding, 2*\squaresize + 2*\padding) rectangle ++(\squaresize,\squaresize);
    \draw[line width=0.5pt, fill=black!30] (1*\squaresize + 1*\padding, 2*\squaresize + 2*\padding) rectangle ++(\squaresize,\squaresize);
    \draw[line width=0.5pt, fill=black!80] (2*\squaresize + 2*\padding, 2*\squaresize + 2*\padding) rectangle ++(\squaresize,\squaresize);
\end{tikzpicture}
\qquad
\begin{tikzpicture}
    \draw[line width=0.5pt, fill=white]    (0*\squaresize + 0*\padding, 0*\squaresize + 0*\padding) rectangle ++(\squaresize,\squaresize);
    \draw[line width=0.5pt, fill=white]    (1*\squaresize + 1*\padding, 0*\squaresize + 0*\padding) rectangle ++(\squaresize,\squaresize);
    \draw[line width=0.5pt, fill=black!30] (2*\squaresize + 2*\padding, 0*\squaresize + 0*\padding) rectangle ++(\squaresize,\squaresize);
    \draw[line width=0.5pt, fill=black!80]    (0*\squaresize + 0*\padding, 1*\squaresize + 1*\padding) rectangle ++(\squaresize,\squaresize);
    \draw[line width=0.5pt, fill=black!80] (1*\squaresize + 1*\padding, 1*\squaresize + 1*\padding) rectangle ++(\squaresize,\squaresize);
    \draw[line width=0.5pt, fill=white]    (2*\squaresize + 2*\padding, 1*\squaresize + 1*\padding) rectangle ++(\squaresize,\squaresize);
    \draw[line width=0.5pt, fill=black!80] (0*\squaresize + 0*\padding, 2*\squaresize + 2*\padding) rectangle ++(\squaresize,\squaresize);
    \draw[line width=0.5pt, fill=black!80] (1*\squaresize + 1*\padding, 2*\squaresize + 2*\padding) rectangle ++(\squaresize,\squaresize);
    \draw[line width=0.5pt, fill=white] (2*\squaresize + 2*\padding, 2*\squaresize + 2*\padding) rectangle ++(\squaresize,\squaresize);
\end{tikzpicture}
\end{center}
Note that this coloring is in fact a Ramsey (or Rado) coloring of $\FF_3^3$, meaning that it avoids any non-degenerate monochromatic copies of $3$-APs. The matching lower bound in fact implies that no such coloring can exist of $\FF_3^4$. This kind of connection between the Ramsey multiplicity problem and Ramsey numbers has previously been explored by Lidick{\'y} and Pfender~\cite{lidicky2021semidefinite}.

\paragraph{Lower bound.} Let us again define some variable names for the relevant flags corresponding to types of dimension $-1$ and $0$ in a graphical way:
\medskip
\begin{center}
    \begin{tabular}{cl}
        \multicolumn{2}{c}{Flags of type $\varnothing$} \bigskip \\
        $F_{1}$ & \drawsquares{white, white, white} \medskip \\
        $F_{2}$ & \drawsquares{black!30, black!30, black!30} \medskip \\
        $F_{3}$ & \drawsquares{black!80, black!80, black!80} \medskip \\\medskip  \\\medskip  \\
    \end{tabular}
    \qquad
    \begin{tabular}{cl}
        \multicolumn{2}{c}{Flags of type \scalebox{0.6}{\drawemphsquares{white}}} \bigskip \\
        $F_{1,1}$ & \drawemphsquares{white, white, white} \medskip \\
        $F_{1,2}$ & \drawemphsquares{white, black!30, black!30} \medskip \\
        $F_{1,3}$ & \drawemphsquares{white, white, black!30} \medskip \\
        $F_{1,4}$ & \drawemphsquares{white, black!30, black!80} \medskip \\
        $F_{1,5}$ & \drawemphsquares{white, white, black!80} 
    \end{tabular}
    \qquad
    \begin{tabular}{cl}
        \multicolumn{2}{c}{Flags of type \scalebox{0.6}{\drawemphsquares{black!30}}} \bigskip \\
        $F_{2,1}$ & \drawemphsquares{black!30, black!30, black!30} \medskip \\
        $F_{2,2}$ & \drawemphsquares{black!30, black!80, black!80} \medskip \\
        $F_{2,3}$ & \drawemphsquares{black!30, black!30, black!80} \medskip \\
        $F_{2,4}$ & \drawemphsquares{black!30, black!80, white} \medskip \\
        $F_{2,5}$ & \drawemphsquares{black!30, black!30, white} 
    \end{tabular}
    \qquad
    \begin{tabular}{cl}
        \multicolumn{2}{c}{Flags of type \scalebox{0.6}{\drawemphsquares{black!80}}} \bigskip \\
        $F_{3,1}$ & \drawemphsquares{black!80, black!80, black!80} \medskip  \\
        $F_{3,2}$ & \drawemphsquares{black!80, white, white} \medskip \\
        $F_{3,3}$ & \drawemphsquares{black!80, black!80, white} \medskip  \\
        $F_{3,4}$ & \drawemphsquares{black!80, white, black!30} \medskip \\
        $F_{3,5}$ & \drawemphsquares{black!80, black!80, black!30} 
    \end{tabular}
    \qquad
\end{center}
\medskip
A proof of the lower bound now follows by computationally verifying that 
\begin{align*}
    F_i - 1/27 & \geq 26/27 \cdot \big\llbracket \big( F_{i,1} -99/182 \, F_{i,2} + 75/208 \, F_{i,3} -11/28 \, F_{i,4} -3/26 \, F_{i,5} \big)^2 \big\rrbracket_{-1} \\
    & \phantom{\geq} \quad \ldots + 1685/1911 \cdot \big\llbracket \big( F_{i,2} - 231/26960 \, F_{i,3} + 1703/6740 \,  F_{i,4} -1869/3370 \, F_{i,5}  \big)^2 \big\rrbracket_{-1} \\
    & \phantom{\geq} \quad \ldots + 71779/431360 \cdot \big\llbracket \big( F_{i,3} -358196/502453 \, F_{i,4} -412904/502453 \, F_{i,5}  \big)^2 \big\rrbracket_{-1} \\
    & \phantom{\geq} \quad \ldots + 5431408/10551513 \cdot \big\llbracket \big( F_{i,4} - 1/4 \, F_{i,5}  \big)^2 \big\rrbracket_{-1}
\end{align*}
holds over all elements in $\Gamma_{3,3}^{-1}(2)$ for any $i \in \{1, 2, 3\}$, that is all unfixed $3$-colorings of $2$-dimensions. Obtaining all $140$ such colorings is computationally straightforward.

\subsection{Proofs of remaining statements}

\paragraph{$5$-APs in $\FF_5^n$. }  The $1/126$ upper bound for $m_{5,2}(A_{5\textrm{-AP}})$ is derived from an iterative blow-up by applying \cref{prop:it_ub} to the same 3-dimensional $2$-coloring as used for the $m_{5,2}(A_{4\textrm{-AP}})$ upper bound. This overlapping use of constructions is also sometimes observed in graph theory, see for example some of the results in~\cite{PikhurkoVaughan_2013, parczyk2022new}.

\paragraph{Schur triple in $\FF_2^n$.} The upper bound of $1/16$ for $m_{2,3}(A_\text{Schur})$ is obtained from a simple blow-up, that is by applying \cref{lemma:blow_ub} to the following $2$-dimensional $3$-coloring:
\medskip
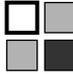
\begin{center}
\begin{tikzpicture}
    \draw[line width=0.5pt, fill=black!30]    (0*\squaresize + 0*\padding, 0*\squaresize + 0*\padding) rectangle ++(\squaresize,\squaresize);
    \draw[line width=0.5pt, fill=black!80]    (1*\squaresize + 1*\padding, 0*\squaresize + 0*\padding) rectangle ++(\squaresize,\squaresize);
    \draw[line width=1.5pt, fill=white]    (0*\squaresize + 0*\padding, 1*\squaresize + 1*\padding) rectangle ++(\squaresize,\squaresize);
    \draw[line width=0.5pt, fill=black!30] (1*\squaresize + 1*\padding, 1*\squaresize + 1*\padding) rectangle ++(\squaresize,\squaresize);
\end{tikzpicture}
\end{center}
The lower bound of $0.041258$ was obtained using the flag algebra framework over all $6546$ colorings of dimension $N = 4$. Since no precise target was met and the number of types and flags involved is too large, we do not reproduce the certificate at this point and instead point to the code associated with this paper.

\medskip

\paragraph{Schur triple in $\FF_3^n$.} The upper bound of $7/81$ for $m_{3,3} (A_\text{Schur})$ is obtained from a simple blow-up, that is by applying \cref{lemma:blow_ub} to the following $2$-dimensional $3$-coloring:
\medskip
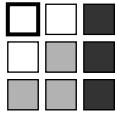
\begin{center}
\begin{tikzpicture}
    \draw[line width=0.5pt, fill=black!30]    (0*\squaresize + 0*\padding, 0*\squaresize + 0*\padding) rectangle ++(\squaresize,\squaresize);
    \draw[line width=0.5pt, fill=black!30]    (1*\squaresize + 1*\padding, 0*\squaresize + 0*\padding) rectangle ++(\squaresize,\squaresize);
    \draw[line width=0.5pt, fill=black!80] (2*\squaresize + 2*\padding, 0*\squaresize + 0*\padding) rectangle ++(\squaresize,\squaresize);
    \draw[line width=0.5pt, fill=white]    (0*\squaresize + 0*\padding, 1*\squaresize + 1*\padding) rectangle ++(\squaresize,\squaresize);
    \draw[line width=0.5pt, fill=black!30] (1*\squaresize + 1*\padding, 1*\squaresize + 1*\padding) rectangle ++(\squaresize,\squaresize);
    \draw[line width=0.5pt, fill=black!80]    (2*\squaresize + 2*\padding, 1*\squaresize + 1*\padding) rectangle ++(\squaresize,\squaresize);
    \draw[line width=1.5pt, fill=white] (0*\squaresize + 0*\padding, 2*\squaresize + 2*\padding) rectangle ++(\squaresize,\squaresize);
    \draw[line width=0.5pt, fill=white] (1*\squaresize + 1*\padding, 2*\squaresize + 2*\padding) rectangle ++(\squaresize,\squaresize);
    \draw[line width=0.5pt, fill=black!80] (2*\squaresize + 2*\padding, 2*\squaresize + 2*\padding) rectangle ++(\squaresize,\squaresize);
\end{tikzpicture}
\end{center}
We note that there are other $2$-dimensional $3$-coloring giving the same value and that we do not believe this upper bound to be tight, but did not spend any significant time and resources trying to improve the construction.

\section{Concluding Remarks and Open Problems} \label{sec:remarks}

There is a long history of applying tools and techniques from graph theory to less studied problems in additive combinatorics. Adapting them to fit a slightly different context not only advances our understanding in the additive setting but also potentially enhances our general knowledge of these tools.
In this paper, we studied an additive analogue to the Ramsey multiplicity problem in vector spaces over finite fields, where our approach involved using blow-up type constructions and developing the flag algebra method for this specific setting. Key aspects included establishing an appropriate notion of when colorings of vector spaces over finite fields are isomorphic, considering the potential invariance of the linear system, and defining a natural notion of solutions amenable to double counting.

There are numerous avenues to explore that would build on these results. First of all, it would obviously  be of interest to try and obtain a lower bound matching the constructive upper bound in \cref{prop:4apresult} as well as stability results in \cref{thm:3ap3colorresult}. Going beyond the problems studied here, it may be of interested to study the multiplicity of higher dimensional affine subspaces, since $q$-AP in $\FF_q^n$ form $1$-dimensional affine subspaces. Developing different but related notions of morphisms for the grid $[q]^n$ could also enable us to apply the framework both to more geometric problems as well as to problems concerning vertex-colorings of the hypercube. A more ambitious goal would be to extend the flag algebra framework to additive structures in other groups where substructures are not as clearly defined and well-behaved as in vector spaces over finite fields.

We note that the SDP-based flag algebra approach however faces a much steeper hurdle compared to similar problems in graph theory: since the number of underlying objects grows exponentially rather than polynomially (respectively in the dimension and the order), one very quickly faces severe computational limitations. At the same time, this lead to the certificates presented in this paper being compact enough (with respect to the number of variables but not the number of constraints) to be presented in written form. Rounding the results of the SDP solver, commonly one of the biggest technical hurdles with these types of proofs, likewise was reasonably straight-forward due to this reason, even when no matching upper bound was known. Nevertheless, we think it is worthwhile to continue to develop this framework for problems beyond graph theory.

\bigskip

\noindent \textbf{Acknowledgements.} We would like to thank Oriol Serra for helpful discussions concerning existing results in the additive setting as well as Pascal Schweitzer for valuable input on the isomorphism problem. This work was partially funded by the Deutsche Forschungsgemeinschaft (DFG, German Research Foundation) under Germany's Excellence Strategy – The Berlin Mathematics Research Center MATH+ (EXC-2046/1, project ID: 390685689), the grants MTM2017-82166-P, PID2020-113082GBI00, and the Severo Ochoa and Mar{\'i}a de Maeztu Program for Centers and Units of Excellence in R\&D (CEX2020-001084-M).


\begin{thebibliography}{10}

\bibitem{baber2011hypergraphs}
R.~Baber and J.~Talbot.
\newblock Hypergraphs do jump.
\newblock {\em Combinatorics, Probability and Computing}, 20(2):161--171, 2011.

\bibitem{balogh2022spectrum}
J.~Balogh, F.~C. Clemen, and B.~Lidick{\`y}.
\newblock The spectrum of triangle-free graphs.
\newblock {\em arXiv:2204.00093}, 2022.

\bibitem{balogh2017rainbow}
J.~Balogh, P.~Hu, B.~Lidick{\`y}, F.~Pfender, J.~Volec, and M.~Young.
\newblock Rainbow triangles in three-colored graphs.
\newblock {\em Journal of Combinatorial Theory, Series B}, 126:83--113, 2017.

\bibitem{csdp}
B.~Borchers.
\newblock {CSDP}, {AC} library for semidefinite programming.
\newblock {\em Optimization methods and Software}, 11(1-4):613--623, 1999.

\bibitem{CameronCillerueloSerra_2007}
P.~J. Cameron, J.~Cilleruelo, and O.~Serra.
\newblock On monochromatic solutions of equations in groups.
\newblock {\em Revista Matem{\'a}tica Iberoamericana}, 23(1):385--395, 2007.

\bibitem{conlon2007ramsey}
D.~Conlon.
\newblock On the {R}amsey multiplicity of complete graphs.
\newblock {\em arXiv preprint arXiv:0711.4999}, 2007.

\bibitem{CummingsEtAl_2013}
J.~Cummings, D.~Kr{\'a}l', F.~Pfender, K.~Sperfeld, A.~Treglown, and M.~Young.
\newblock Monochromatic triangles in three-coloured graphs.
\newblock {\em Journal of Combinatorial Theory, Series B}, 103(4):489--503,
  2013.

\bibitem{DasEtAl_2013}
S.~Das, H.~Huang, J.~Ma, H.~Naves, and B.~Sudakov.
\newblock A problem of {E}rd{\H{o}}s on the minimum number of $k$-cliques.
\newblock {\em Journal of Combinatorial Theory, Series B}, 103(3):344--373,
  2013.

\bibitem{Datskovsky_2003}
B.~A. Datskovsky.
\newblock On the number of monochromatic {S}chur triples.
\newblock {\em Advances in Applied Mathematics}, 31(1):193--198, 2003.

\bibitem{SilvaEtAl_2016}
M.~K. de~Carli~Silva, F.~M. de~Oliveira~Filho, and C.~M. Sato.
\newblock Flag algebras: a first glance.
\newblock {\em arXiv preprint arXiv:1607.04741}, 2016.

\bibitem{Erdos_1962}
P.~Erd\H{o}s.
\newblock On the number of complete subgraphs contained in certain graphs.
\newblock {\em Magyar Tud. Akad. Mat. Kutat{\'o} Int. K{\"o}zl}, 7(3):459--464,
  1962.

\bibitem{even2015note}
C.~Even-Zohar and N.~Linial.
\newblock A note on the inducibility of 4-vertex graphs.
\newblock {\em Graphs and Combinatorics}, 31(5):1367--1380, 2015.

\bibitem{FalgasVaughan_2013}
V.~Falgas-Ravry and E.~R. Vaughan.
\newblock Applications of the semi-definite method to the {T}ur{\'a}n density
  problem for 3-graphs.
\newblock {\em Combinatorics, Probability and Computing}, 22(1):21--54, 2013.

\bibitem{FoxPhamZhao_2021}
J.~Fox, H.~T. Pham, and Y.~Zhao.
\newblock Common and {S}idorenko linear equations.
\newblock {\em The Quarterly Journal of Mathematics}, 72(4):1223--1234, 2021.

\bibitem{FranekRodl_1993}
F.~Franek and V.~R{\"o}dl.
\newblock 2-colorings of complete graphs with a small number of monochromatic
  {$K_4$} subgraphs.
\newblock {\em Discrete mathematics}, 114(1-3):199--203, 1993.

\bibitem{GamrathAndersonBestuzhevaetal.2020}
G.~Gamrath, D.~Anderson, K.~Bestuzheva, W.-K. Chen, L.~Eifler, M.~Gasse,
  P.~Gemander, A.~Gleixner, L.~Gottwald, K.~Halbig, G.~Hendel, C.~Hojny,
  T.~Koch, P.~L. Bodic, S.~J. Maher, F.~Matter, M.~Miltenberger, E.~M{\"u}hmer,
  B.~M{\"u}ller, M.~Pfetsch, F.~Schl{\"o}sser, F.~Serrano, Y.~Shinano,
  C.~Tawfik, S.~Vigerske, F.~Wegscheider, D.~Weninger, and J.~Witzig.
\newblock The scip optimization suite 7.0.
\newblock Technical Report 20-10, ZIB, Takustr. 7, 14195 Berlin, 2020.

\bibitem{Giraud_1979}
G.~Giraud.
\newblock Sur le probleme de {G}oodman pour les quadrangles et la majoration
  des nombres de {R}amsey.
\newblock {\em Journal of Combinatorial Theory, Series B}, 27(3):237--253,
  1979.

\bibitem{GleixnerSteffyWolter2015}
A.~Gleixner, D.~Steffy, and K.~Wolter.
\newblock Iterative refinement for linear programming.
\newblock Technical Report 15-15, ZIB, Takustr. 7, 14195 Berlin, 2015.

\bibitem{Goodman_1959}
A.~W. Goodman.
\newblock On sets of acquaintances and strangers at any party.
\newblock {\em The American Mathematical Monthly}, 66(9):778--783, 1959.

\bibitem{GrahamRodlRucinski_1996}
R.~Graham, V.~R{\"o}dl, and A.~Ruci{\'n}ski.
\newblock On {S}chur properties of random subsets of integers.
\newblock {\em Journal of Number Theory}, 61(2):388--408, 1996.

\bibitem{grzesik2012maximum}
A.~Grzesik.
\newblock On the maximum number of five-cycles in a triangle-free graph.
\newblock {\em Journal of Combinatorial Theory, Series B}, 102(5):1061--1066,
  2012.

\bibitem{GrzesikEtAl_2020}
A.~Grzesik, J.~Lee, B.~Lidick{\`y}, and J.~Volec.
\newblock On tripartite common graphs.
\newblock {\em arXiv preprint arXiv:2012.02057}, 2020.

\bibitem{KamcevLiebenauMorrison_2021b}
N.~Kamcev, A.~Liebenau, and N.~Morrison.
\newblock On uncommon systems of equations.
\newblock {\em arXiv preprint arXiv:2106.08986}, 2021.

\bibitem{KamcevLiebenauMorrison_2021a}
N.~Kam{\v{c}}ev, A.~Liebenau, and N.~Morrison.
\newblock Towards a characterisation of {S}idorenko systems.
\newblock {\em arXiv preprint arXiv:2107.14413}, 2021.

\bibitem{lamaison2022common}
A.~Lamaison, P.~P. Pach, et~al.
\newblock Common systems of two equations over the binary field.
\newblock In {\em Discrete Mathematics Days 2022}, pages 169--173. Editorial de
  la Universidad de Cantabria, 2022.

\bibitem{lidicky2021semidefinite}
B.~Lidicky and F.~Pfender.
\newblock Semidefinite programming and {R}amsey numbers.
\newblock {\em SIAM Journal on Discrete Mathematics}, 35(4):2328--2344, 2021.

\bibitem{LuPeng_2012}
L.~Lu and X.~Peng.
\newblock Monochromatic $4$-term arithmetic progressions in $2$-colorings of
  $\mathbb{Z}_n$.
\newblock {\em Journal of Combinatorial Theory, Series A}, 119(5):1048--1065,
  2012.

\bibitem{McKay_1981}
B.~D. McKay.
\newblock Practical graph isomorphism.
\newblock {\em Congressus Numerantium}, 30:45--87, 1981.

\bibitem{McKay_1998}
B.~D. McKay.
\newblock Isomorph-free exhaustive generation.
\newblock {\em Journal of Algorithms}, 26(2):306--324, 1998.

\bibitem{McKayPiperno_2014}
B.~D. McKay and A.~Piperno.
\newblock Practical graph isomorphism, {II}.
\newblock {\em Journal of symbolic computation}, 60:94--112, 2014.

\bibitem{Niess_2012}
S.~Nie{\ss}.
\newblock Counting monochromatic copies of {$K_4$}: a new lower bound for the
  {R}amsey multiplicity problem.
\newblock {\em arXiv:1207.4714}, 2012.

\bibitem{Nikiforov_2001}
V.~Nikiforov.
\newblock On the minimum number of $k$-cliques in graphs with restricted
  independence number.
\newblock {\em Combinatorics, Probability and Computing}, 10(4):361--366, 2001.

\bibitem{parczyk2022new}
O.~Parczyk, S.~Pokutta, C.~Spiegel, and T.~Szab{\'o}.
\newblock New {R}amsey multiplicity bounds and search heuristics.
\newblock {\em Discrete Mathematics Days 2022}, 263:224, 2022.

\bibitem{parrilo2008asymptotic}
P.~A. Parrilo, A.~Robertson, and D.~Saracino.
\newblock On the asymptotic minimum number of monochromatic 3-term arithmetic
  progressions.
\newblock {\em Journal of Combinatorial Theory, Series A}, 115(1):185--192,
  2008.

\bibitem{parrilo2003minimizing}
P.~A. Parrilo and B.~Sturmfels.
\newblock Minimizing polynomial functions.
\newblock In S.~Basu and L.~Gonz{\'a}lez-Vega, editors, {\em Algorithmic and
  Quantitative Real Algebraic Geometry}, volume~60 of {\em DIMACS Series in
  Discrete Mathematics and Theoretical Computer Science}, Providence, RI, 2003.
  American Mathematical Society.

\bibitem{PikhurkoVaughan_2013}
O.~Pikhurko and E.~R. Vaughan.
\newblock Minimum number of $k$-cliques in graphs with bounded independence
  number.
\newblock {\em Combinatorics, Probability and Computing}, 22(6):910--934, 2013.

\bibitem{rado1933studien}
R.~Rado.
\newblock Studien zur kombinatorik.
\newblock {\em Mathematische Zeitschrift}, 36:424–--470, 1933.

\bibitem{razborov2007flag}
A.~A. Razborov.
\newblock Flag algebras.
\newblock {\em The Journal of Symbolic Logic}, 72(4):1239--1282, 2007.

\bibitem{razborov20103}
A.~A. Razborov.
\newblock On 3-hypergraphs with forbidden 4-vertex configurations.
\newblock {\em SIAM Journal on Discrete Mathematics}, 24(3):946--963, 2010.

\bibitem{RobertsonZeilberger_1998}
A.~Robertson and D.~Zeilberger.
\newblock A 2-coloring of $[1, n] $ can have $(1/22) n^2 + o(n) $ monochromatic
  {S}chur triples, but not less!
\newblock {\em The Electronic Journal of Combinatorics}, 5(1):R19, 1998.

\bibitem{SaadWolf_2017}
A.~Saad and J.~Wolf.
\newblock {R}amsey multiplicity of linear patterns in certain finite abelian
  groups.
\newblock {\em The Quarterly Journal of Mathematics}, 68(1):125--140, 2017.

\bibitem{Schoen_1999}
T.~Schoen.
\newblock The number of monochromatic {S}chur triples.
\newblock {\em European Journal of Combinatorics}, 20(8):855--866, 1999.

\bibitem{Sperfeld_2011}
K.~Sperfeld.
\newblock On the minimal monochromatic {$K_4$}-density.
\newblock {\em arXiv preprint arXiv:1106.1030}, 2011.

\bibitem{Thomason_1989}
A.~Thomason.
\newblock A disproof of a conjecture of {E}rd{\H{o}}s in {R}amsey theory.
\newblock {\em Journal of the London Mathematical Society}, 2(2):246--255,
  1989.

\bibitem{Thomason_1997}
A.~Thomason.
\newblock Graph products and monochromatic multiplicities.
\newblock {\em Combinatorica}, 17(1):125--134, 1997.

\bibitem{wolf2010minimum}
J.~Wolf.
\newblock The minimum number of monochromatic 4-term progressions in
  $\mathbb{Z}_p$.
\newblock {\em Journal of Combinatorics}, 1(1):53--68, 2010.

\end{thebibliography}
\end{document}